\documentclass[reqno,11pt]{amsart}

\usepackage{amsmath,amsfonts,amssymb,amsthm,epsfig}
\usepackage{mathtools}
\usepackage[english]{babel}
\usepackage{tikz}
\usetikzlibrary{arrows}
\usetikzlibrary {patterns}
\usepackage{comment}
\usepackage{enumitem}
\usepackage[colorlinks=true,urlcolor=blue,citecolor=blue,linkcolor=blue,linktocpage,pdfpagelabels,bookmarksnumbered,bookmarksopen]{hyperref}
 \allowdisplaybreaks
\numberwithin{equation}{section}

\font\script=rsfs10 at 11pt
\def\eps{\varepsilon}
\def\H{{\mbox{\script H}\,\,}}
\def\R{\mathbb R}
\def\C{\mathcal C}
\def\S{\mathbb S}
\def\N{\mathbb N}
\def\bal{\begin{aligned}}
\def\eal{\end{aligned}}
\def\proofof#1{\begin{proof}[Proof of #1]}
\def\Chi#1{\hbox{{\large $\chi$}{\Large $_{_{#1}}$}}}

\def\step#1#2{\par\noindent{\underline{\it Step~#1.}}\emph{ #2}\\}
\def\freccia#1{\xrightarrow[\ #1]{}}
\def\XXint#1#2#3{{\setbox0=\hbox{$#1{#2#3}{\int}$} \vcenter{\vspace{-1pt}\hbox{$#2#3$}}\kern-.5\wd0}}

\pgfarrowsdeclare{<<<}{>>>}
{
\arrowsize=0.2pt
\advance\arrowsize by .5\pgflinewidth
\pgfarrowsleftextend{-4\arrowsize-.5\pgflinewidth}
\pgfarrowsrightextend{.5\pgflinewidth}
}
{
\arrowsize=0.2pt
\advance\arrowsize by .5\pgflinewidth
\pgfpathmoveto{\pgfpointorigin}
\pgfpathlineto{\pgfpoint{-1.2mm}{-.8mm}}
\pgfusepathqstroke
\pgfpathmoveto{\pgfpointorigin}
\pgfpathlineto{\pgfpoint{-1.2mm}{.8mm}}
\pgfusepathqstroke
}

\newcounter{mt}
\def\maintheorem#1#2#3{\par \medskip \noindent {\bf Theorem~\mref{#1}}~(#2).~{\it #3}\par}
\def\mref#1{\Alph{#1}}
\def\maintheoremdeclaration#1{\stepcounter{mt}\newcounter{#1}\setcounter{#1}{\arabic{mt}}}

\maintheoremdeclaration{nobound}
\maintheoremdeclaration{approx1}
\maintheoremdeclaration{approx2}

\newtheorem{theorem}{Theorem}[section]
\newtheorem{lemma}[theorem]{Lemma}
\newtheorem{prop}[theorem]{Proposition}
\newtheorem{corollary}[theorem]{Corollary}
\newtheorem{remark}[theorem]{Remark}
\theoremstyle{definition}
\newtheorem{defin}[theorem]{Definition}
\newtheorem{example}[theorem]{Example}
\title[On the approximation of finite perimeter sets]{On the approximation of finite perimeter sets}

\author[A.~Carbotti]{A.~Carbotti}
\address[A.~Carbotti]{Dipartimento di Matematica e Fisica ``E.\ De Giorgi'', Universit\`a del Salento, Via per Arnesano, 73100 Lecce (LE), Italy}
\email{alessandro.carbotti@unisalento.it}

\author[S.~Cito]{S.~Cito}
\address[S.~Cito]{Dipartimento di Matematica e Fisica ``E.\ De Giorgi'', Universit\`a del Salento, Via per Arnesano, 73100 Lecce (LE), Italy}
\email{simone.cito@unisalento.it}

\author[D.~A.~La Manna]{D.~A.~La Manna}
\address[D.~A.~La Manna]{Dipartimento di Matematica e Applicazioni ``R.\ Caccioppoli'', Universit\`a  degli Studi di Napoli ``Federico II'', Via Cintia, Monte S.\ Angelo, 80126 Napoli (NA), Italy}
\email{domenicoangelo.lamanna@unina.it}

\author[A.~Pratelli]{A.~Pratelli}
\address[A.~Pratelli]{Dipartimento di Matematica, Universit\`a  di Pisa, Largo B. Pontecorvo 5, 56127 Pisa (PI), Italy}
\email{aldo.pratelli@unipi.it}

\author[G.~Stefani]{G.~Stefani}
\address[G.~Stefani]{Dipartimento di Matematica ``Tullio Levi-Civita'', Universit\`a di Padova, via Trieste 63, 35121 Padova (PD), Italy}
\email{giorgio.stefani@unipd.it}

\begin{document}

\begin{abstract}
We prove that if $\Omega\subseteq\mathbb{R}^N$ is a set with finite perimeter with $\H^{N-1}(\partial \Omega\setminus\partial^* \Omega)=0$, then any set of finite perimeter $E\subseteq\mathbb{R}^N$ can be approximated by a polyhedral or smooth bounded set $F$ in such a way that both the total perimeter of $E$ and the perimeter of $E$ inside $\Omega$ are approximated by those of $F$, and the boundary of $F$ has negligible intersection with the boundary of $\Omega$.  
In addition, we address the approximation for perimeter and volume with densities, and we present counterexamples illustrating the sharpness of our assumptions.  
Our constructions rely on a technical result that replaces $E$ with a set $F$ which agrees with $E$ and has the same boundary inside $\Omega$, while sharing no common boundary with $\Omega$, and does so without substantially altering the perimeter or the volume of the original set.
\end{abstract}

\maketitle

\section{Introduction}

We investigate the problem of approximating a given set $E$ by a polyhedral or smooth set $F$, in the sense of perimeter, inside a prescribed ``container'' $\Omega$. 
Beyond the standard approximation theory (see, for instance, \cite[Theorem 3.42 and Remark 3.42]{AFP}), our starting point is the following result; see \cite[Proposition~15]{ADM} and \cite[Theorem~A.4]{CS}, and the discussion therein.

\begin{theorem}[Approximation with a Lipschitz set $\Omega$]\label{startingpoint}
Let $\Omega\subseteq\R^N$ be a bounded open set with Lipschitz boundary, and let $E\subseteq\R^N$ be a set of finite perimeter in $\Omega$. Then, for every $\eps>0$ there exists a polyhedral set $F$ such that 
\begin{equation*}
|P(F;\Omega)-P(E;\Omega)|<\eps,
\qquad
|F\triangle E|<\eps,
\qquad
\H^{N-1}(\partial^* F\cap \partial^*\Omega)=0.
\end{equation*}
\end{theorem}

Approximations of this type play a fundamental role in a variety of settings. Indeed, many variational problems rely on the possibility of replacing sets of finite perimeter with smooth or polyhedral ones.  
Here we only mention the $\Gamma$-convergence results contained in~\cite{ADM,CCLP,CS,DRLM}, which were the main motivation for the present work, although the list could easily be extended.  

Our aim is to extend Theorem~\ref{startingpoint} in several directions. First, we wish to weaken the regularity assumptions on $\Omega$. Second, we aim to estimate the perimeter of $E$ also outside $\Omega$. More precisely, Theorem~\ref{startingpoint} only concerns the set $E$ inside $\Omega$; thus, replacing $E$ by $E\cap\Omega$ (or directly assuming that $E\subseteq\Omega$) makes no difference for its conclusion. However, it may be relevant to deal with a set $E$ of finite perimeter in the whole $\mathbb{R}^N$, and to approximate it by a polyhedral or smooth set $F$ having no boundary in common with $\partial\Omega$, in such a way that $P(F;\Omega)$ and $P(F)$ are $\varepsilon$-close to $P(E;\Omega)$ and $P(E)$, respectively.
Third, we would like to treat the more general situation in which the notions of volume and perimeter in $\mathbb{R}^N$ are defined through densities. More precisely, given two Borel functions $f:\mathbb{R}^N \to (0,+\infty)$ and $g:\mathbb{R}^N \times \mathbb{S}^{N-1} \to (0,+\infty)$, we define the \emph{$f$-volume} and the \emph{$g$-perimeter} of a set $E$ as
\begin{align}\label{genvolper}
|E|_f = \int_E f(x)\,dx\,, && P_g(E) = \int_{\partial^* E} g(x,\nu_E(x))\,d\H^{N-1}(x)\,,
\end{align}
which reduce to the standard ones when $f\equiv 1$ and $g\equiv 1$. We say that $E$ is \emph{a set of finite $g$-perimeter} if it is a set of locally finite Euclidean perimeter, and $P_g(E)<+\infty$. Our aim is to obtain a very general extension of Theorem~\ref{startingpoint} in all three of the above directions.

Our main results are obtained only after performing a first, crucial modification of the set $E$, which allows us to remove its common boundary with the container. More precisely, given a set $E$ of finite perimeter such that $\H^{N-1}(\partial^{*}E \cap \partial^{*}\Omega) > 0$, we aim to modify $E$ \emph{only outside} $\Omega$ so as to eliminate this common boundary without significantly altering either the perimeter or the volume.
This construction contains the main technical difficulties of the paper, and we believe it is of independent interest. Here, and in the rest of the paper, whenever $G$ and $H$ are two sets of finite perimeter, we let $P(G;H)=\H^{N-1}(\partial^* G \cap H)$.

\maintheorem{nobound}{Removing the common boundary}
{Let $\Omega$ and $E$ be two sets of locally finite perimeter in $\R^N$. Then, for every $\eps>0$ there exists a set $F\subseteq\R^N$ such that
\begin{gather}\label{eq:A.1}
F\cap \Omega= E\cap \Omega\,,\qquad\partial^* F \cap \Omega=\partial^* E \cap \Omega\,,
\\\label{eq:A.2}
 |F\triangle E| < \eps\,,\\
\label{eq:A.3}
|P(F) - P(E)|< \eps\,,\\
\label{eq:A.4}
\H^{N-1}(\partial^* F \cap \partial^* \Omega)=0\,.
\end{gather}}

\begin{figure}[h!]
\begin{tikzpicture}
\filldraw[fill=green!25, line width=.5pt] (3.5,-4) .. controls (2,-3.5) and (1,-5) .. (0.5,-5) .. controls (0,-5) and (-1,-3.5) .. (-1,-3) .. controls (-1,-2.5) and (0,-1.5) .. (0,-1) .. controls (0,-.5) and (-1,1.5) .. (-1,2) .. controls (-1,2.5) and (-.5,3) .. (0,3) .. controls (.5,3) and (1.8,2) .. (2,2) .. controls (3,2) and (3.6,1.5) .. (4,1.3) .. controls (4.4,1.1) and (4.8,1.5) .. (5,2) .. controls (5.2,2.5) and (6,3) .. (6.5,2.5) .. controls (6.7,2) and (5.2,0.5) .. (4.5,0.2) .. controls (3.8,-.1) and (2,0.2) .. (2,-.5) .. controls (2,-1.2) and (3,-3.5) .. (3.5,-3) .. controls (4,-2.5) and (4.2,-2.5) .. (5,-2) .. controls (5.4,-1.75) and (6,-2) .. (6.5,-2.5) .. controls (7,-3) and (5,-4.5) .. (3.5,-4);
\draw[line width=.9pt] (-2,0) .. controls (-2,1) and (1,2) .. (2,2) .. controls (3,2) and (3.6,1.5) .. (4,1.3) .. controls (4.4,1.1) and (6,-.2) .. (6,-1) .. controls (6,-1.8) and (5.8,-1.5) .. (5,.-2) .. controls (4.2,-2.5) and (4,-2.5) .. (3.5,-3) .. controls (3,-3.5) and (1,-3.6) .. (0,-3.2) .. controls (-1,-2.8) and (-2,-1) .. (-2,0);
\draw (-2,.5) node[anchor=north east] {$\Omega$};
\draw (7.1,3) node[anchor=north east] {$E$};
\filldraw (3,1.83) circle (1.2pt);
\draw (3,1.83) node[anchor=north east] {$x$};
\filldraw (4.26,-2.42) circle (1.2pt);
\draw (4.26,-2.42) node[anchor=south east] {$y$};
\draw (1.4,3.1) node[anchor=north east] {$F$};
\draw[pattern=north east lines, line width=.5pt] (3.5,-4) .. controls (2,-3.5) and (1,-5) .. (0.5,-5) .. controls (0,-5) and (-1,-3.5) .. (-1,-3) .. controls (-1,-2.5) and (0,-1.5) .. (0,-1) .. controls (0,-.5) and (-1,1.5) .. (-1,2) .. controls (-1,2.5) and (-.5,3) .. (0,3) .. controls (.5,3) and (1.8,2) .. (2,2)  .. controls (2,2.1) and (2,2.2) .. (2.1,2.3) .. controls (2.2,2.4) and (2.3,2.36) .. (2.4,2.35) .. controls (2.5,2.34) and (3,2.25) .. (3.22,2.15) .. controls (3.44,2.05) and (4.04,1.8) .. (4.14,1.65) .. controls (4.24,1.5) and (4.12,1.4).. (4,1.3) .. controls (4.4,1.1) and (4.8,1.5) .. (5,2) .. controls (5.2,2.5) and (6,3) .. (6.5,2.5) .. controls (6.7,2) and (5.2,0.5) .. (4.5,0.2) .. controls (3.8,-.1) and (2,0.2) .. (2,-.5) .. controls (2,-1.2) and (3,-3.5) .. (3.5,-3) .. controls (3.58,-3.08) and (3.66,-3.16) .. (3.8,-3.1) .. controls (3.94,-3.04) and (4.14,-2.9) .. (4.34,-2.76) .. controls (4.54,-2.62) and (4.84,-2.4) .. (5.00,-2.3) .. controls (5.16,-2.2) and (5.08,-2.1) .. (5,-2) .. controls (5.4,-1.75) and (6,-2) .. (6.5,-2.5) .. controls (7,-3) and (5,-4.5) .. (3.5,-4);
\end{tikzpicture}
\caption{The green set is $E$, the dashed set is $F$ given by Theorem~\mref{nobound}.}
\label{fig_nobound}
\end{figure}
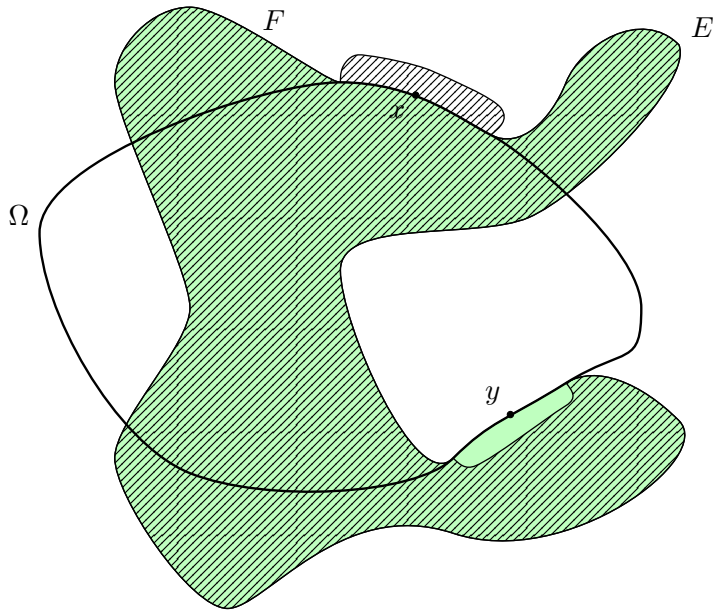

We emphasize that Theorem~\mref{nobound} is not used to replace the set $E$ with a more regular set $F$, but rather to eliminate the common boundary between $E$ and $\Omega$. In particular, the new set $F$ coincides with $E$ inside $\Omega$, and it differs from $E$ only outside $\Omega$, and only when $E$ and $\Omega$ share a boundary of positive measure.
Roughly speaking, as illustrated in Figure~\ref{fig_nobound}, one must ``enlarge'' the set $E$ near portions of the common boundary where $E\subseteq\Omega$ (for instance, near the point $x$ in the figure), and ``shrink'' it near portions where $E$ lies outside $\Omega$ (as near the point $y$). Clearly, if
$\H^{N-1}(\partial^{*}E \cap \partial^{*}\Omega)=0$,
then the conclusion obviously follows by choosing $F=E$.
The possibility of having a non-negligible common boundary between $E$ and $\Omega$ is precisely what makes Theorem~\ref{startingpoint} non-trivial. Likewise, the main difficulty in proving Theorem~\mref{nobound} is the need to ``push'' the common boundary $\partial^{*}E \cap \partial^{*}\Omega$ entirely outside $\Omega$.

We can now state the approximation theorems.
Our first main result is a generalization of Theorem~\ref{startingpoint} in the case $f\equiv1$ and $g\equiv1$. 

\maintheorem{approx1}{Approximation with a general set $\Omega$}{Let $\Omega\subseteq\R^N$ be  a set of finite perimeter such that $\H^{N-1}(\partial \Omega\setminus \partial^*\Omega)=0$. Then, for every set $E\subseteq \R^N$ of finite perimeter in $\R^N$ and for every $\eps>0$ there exists a polyhedral or smooth bounded open set $F$ such that
\begin{gather}
\label{eq:B1}|P(F;\Omega)-P(E;\Omega)|<\eps\,,\qquad|P(F)-P(E)|<\eps\,,\\
\label{eq:B2}|F\triangle E|<\eps\,,\\
\label{eq:B3}  \H^{N-1}(\partial^* F\cap \partial^*\Omega)=0\,.
\end{gather}
}
 
The sole assumption required on $\Omega$ in Theorem~\mref{approx1} is extremely mild, as we only ask that its reduced boundary coincides almost everywhere with its topological one. Moreover, in addition to approximating the perimeter inside $\Omega$, we also obtain the approximation of the total perimeter.

To state our second approximation result, we need some additional notation.

\begin{defin}\label{1homcon}
Let $g:\R^N\times \S^{N-1}\to (0,+\infty)$ be given, and let us call $\tilde g:\R^N\times \R^N\to(0,+\infty)$ the positively $1$-homogeneous extension of $g$; that is, for every $x\in\R^N$, $\nu\in\S^{N-1}$ and $\lambda\geq 0$ we define $\tilde g(x,\lambda \nu) = \lambda g(x,\nu)$. We say that $g$ is \emph{convex in the second variable} whenever so is $\tilde g$.
\end{defin}

The meaning of the above definition is very simple: $g$ is convex in the second variable precisely when, for every $x\in\mathbb{R}^N$, the \emph{unit ball} 
$\{v\in\mathbb{R}^N : g(x,v)=1\}$ is convex. This is a standard assumption, and it is required to guarantee the lower semicontinuity of the perimeter with respect to $L^1$-convergence of sets; namely, if $\chi_{E_j}\to\chi_E$ in $L^1$, then $P_g(E) \leq \liminf_{j} P_g(E_j)$.

With this notation in force, our second main result takes the following form.

\maintheorem{approx2}{Approximation with a general set $\Omega$ and double density}{Let $f:\R^N\to(0,+\infty)$ be $L^1_{\rm loc}(\R^N)$ and locally bounded, and $g:\R^N\times\S^{N-1}\to (0,+\infty)$ continuous in the first variable, and convex in the second one in the sense of Definition~\ref{1homcon}. Let $\Omega\subseteq\R^N$ be a set of finite perimeter such that $\H^{N-1}(\partial\Omega\setminus\partial^* \Omega)=0$. Then, for every set $E\subseteq \R^N$ of finite $g$-perimeter and for every $\eps>0$ there exists a smooth (possibly unbounded) open set $F$ such that
\begin{gather}
|P_g(F;\Omega)-P_g(E;\Omega)|<\eps\,,  \qquad |P_g(F)-P_g(E)|<\eps\,, \label{eq:weighted1}\\
|F\triangle E|_f<\eps\,, \label{eq:weighted2}\\
\H^{N-1}(\partial^* F\cap \partial^*\Omega)=0\,. \label{eq:weighted3}
\end{gather}
The set $F$ can be taken smooth and bounded, or polyhedral and bounded, if $E$ is bounded, or if $|E|_f<\infty$ and there exists a constant $M>0$ such that $g(x,\nu)\le Mf(x)$ for every $x\in\R^N$ and $ \nu\in\S^{N-1}$.
}

\vspace{1ex}

We note that Theorem~\mref{approx2} extends Theorem~\ref{startingpoint} in the third desired direction.
In addition, while the requirement $f\in L^1_{\rm loc}(\mathbb{R}^N)$ is clearly the weakest possible assumption on $f$, the sharpness of the continuity of $g$ is less evident, as well as of the hypotheses listed in the last part of the statement. However, in Section~\ref{sec:counterex} we will show that all these assumptions are indeed sharp.\par

We emphasize that the above results can be extended to the case when $E$ and $\Omega$ are sets of locally finite perimeter (or $g$-perimeter), via standard arguments.

\subsection*{Organization of the paper}
In Section~\ref{subs:not&prel} we fix some notation and recall a few standard results about sets of finite perimeter. Section~\ref{sec:ThA}, which forms the core of the paper, contains the proof of Theorem~\mref{nobound}. In Section~\ref{sec:pfThBeC}, we apply Theorem~\mref{nobound} to obtain a proof of the Euclidean case, Theorem~\mref{approx1}, and of the case with densities $f$ and $g$, thereby proving Theorem~\mref{approx2}. Concerning Theorem~\mref{approx2}, many of the arguments proceed essentially as in the Euclidean setting, and we will emphasize only those steps that require a different treatment. Finally, in Section~\ref{sec:counterex}, we present two counterexamples that illustrate the sharpness of the assumptions in Theorem~\mref{approx2}.

\subsection*{Acknowledgments}
A.~Carbotti,\, A.~Pratelli and G.~Stefani are members of the Istituto Nazionale di Alta Matematica (INdAM) and of the Gruppo Nazionale per l'Analisi Matematica, la Probabilit\`a  e le loro Applicazioni (GNAMPA).

A.~Carbotti has received funding from INdAM under the INdAM--GNAMPA 2026 Project ``Analisi variazionale per operatori locali e nonlocali possibilmente singolari o degeneri'' (grant agreement No. CUP\_E53\-C25\-002\-010\-001).

A.~Carbotti and S.~Cito have received funding from the MUR - PRIN 2022 project ``Elliptic and parabolic problems, heat kernel estimates, and spectral theory'' (N.~20223L2NWK).

A.~Carbotti and G.~Stefani have received funding from INdAM under the INdAM--GNAMPA Project 2025 ``Metodi variazionali per problemi dipendenti da operatori frazionari isotropi e anisotropi'' (grant agreement No.\ CUP\_E53\-240\-019\-500\-01).

S. Cito has received funding from the MUR under the MUR -- PRIN 2022 project ``Stochastic Modeling of Compound Events'' (N.~P2022KZJTZ) in the framework of European Union -- Next Generation EU, and from INdAM under the INdAM--GNAMPA Project 2025 ``Disuguaglianze funzionali di tipo geometrico e spettrale'' (grant agreement No.\ CUP\_E53\-240\-019\-500\-01).

D.~A.~La Manna has received funding from INdAM under the INdAM--GNAMPA Project 2025 ``Local and nonlocal equations with lower order terms'' (grant agreement No.\ CUP\_E53\-240\-019\-500\-01), and from the University of Naples Federico II through FRA Project ``Geometric Topics in Fluid Dynamics''.

A.~Pratelli has received funding from the MUR -- PRIN 2022 project ``Geometric Evolution Problems and Shape Optimization'' (N.~2022E9CF89).

G.~Stefani has received funding from INdAM under the INdAM--GNAMPA 2026 Project ``Metodi non locali classici e distribuzionali per problemi variazionali'' (grant agreement No. CUP\_E53\-C25\-002\-010\-001) and from the European Union -- NextGenerationEU and the University of Padua under the 2023 STARS@UNIPD Starting Grant Project ``New Directions in Fractional Calculus -- NewFrac'' (grant agreement No. CUP\_C95\-F21\-009\-990\-001).

\section{Preliminaries\label{subs:not&prel}}

We list here some standard notation and results about sets of finite perimeter, for all the details one can refer for instance to~\cite{AFP}. A set $\Omega\subseteq\R^N$ is said \emph{of finite perimeter} if its characteristic function $\Chi\Omega$ belongs to $BV(\R^N)$. The \emph{(reduced) boundary} $\partial^* \Omega$ of $\Omega$ is the set of all the points $x\in\R^N$ such that an outer normal vector $\nu=\nu_\Omega(x)$ exists; this means that
\begin{align*}
\lim_{r \searrow 0} \frac{|\Omega \cap B^+_{r,\nu}(x)|}{|B^+_{r,\nu}(x)|} = 0\,, && \lim_{r \searrow 0} \frac{|\Omega \cap B^-_{r,\nu}(x)|}{|B^-_{r,\nu}(x)|} = 1\,,
\end{align*}
where the positive and negative half-balls $B^\pm_{r,\nu}(x)$ are defined as
\[
B^\pm_{r,\nu}(x) = \Big\{ z\in\R^N: |z-x|<r,\, (z-x) \cdot \nu \gtrless 0\Big\}\,.
\]
It can be proved that a set $\Omega$ is of finite perimeter if and only if $\partial^* \Omega$ has finite $\H^{N-1}$-measure; if this is the case, the perimeter of $\Omega$ is defined as $P(\Omega) = \H^{N-1}(\partial^* \Omega)$. It can also be proved that, up to $\H^{N-1}$-negligible subsets, $\partial^* \Omega$ coincides with all the points of density $1/2$ for $\Omega$. Moreover, $\H^{N-1}$-a.e. point of $\R^N$ has density either $0$, or $1/2$, or $1$ with respect to $\Omega$.

It is obvious that, if $\Omega$ is regular (Lipschitz suffices), then the reduced boundary coincides with the topological boundary $\partial\Omega$ up to $\H^{N-1}$-negligible subsets. In general, it is true that $\partial^* \Omega\subseteq\partial \Omega$, while the other inclusion can be false; in fact, the topological boundary of a set of finite perimeter can be much larger than the reduced one, and it can also be the whole $\R^N$. It is important to notice that two sets which coincide $\H^N$-a.e. have the same reduced boundary, but the topological boundaries can be completely different. In Theorem~\mref{approx1} and~\mref{approx2} we require that the topological and reduced boundary of $\Omega$ coincide almost everywhere; notice that this is a very mild regularity assumption on $\Omega$. Indeed, this is true not only in the quite regular case of Lipschitz sets, but, for instance, also for the much less regular case of porous sets.

If $\Omega$ is a set of finite perimeter and $x\in\partial^*\Omega$, then the blow-up properties hold. This means that, calling $\Omega_r = (\Omega-x)/r$, one has that $\Chi{\Omega_r}$ locally converges in the $BV$ sense to $\Chi{\R^{N,-}_\nu}$, where $\R^{N,-}_\nu$ is the half-space $\{z\in\R^N,\, z\cdot \nu<0\}$, being again $\nu=\nu_\Omega(x)$ the outer normal vector.

A set is said \emph{of locally finite perimeter} if its intersection with any ball is a set of finite perimeter.

Let us fix some notation. Given a point $x$ and a vector $\nu\in\S^{N-1}$, for every $r>0$ we let $C(x,r)$ be the cylinder centered at $x$, with radius $r$, height $2r$ and axis parallel to $\nu$. Moreover, we let $C^\pm(x,r)$ be its upper and lower half, that is, the points $z\in C(x,r)$ such that $(z-x)\cdot \nu$ is respectively positive and negative. In addition, for any $0<\delta\leq 1$ we let $C_{\delta r}(x,r)$ be the cylinder defined in the same way as $C(x,r)$, but with height $2\delta r$ instead of $2r$.

Putting together~\cite[Theorem 2.63]{AFP} and~\cite[Theorem 2.83]{AFP}, we get the following useful result.

\begin{theorem}\label{hnleb}
Let $\Gamma\subseteq\R^N$ be a set with $\H^{N-1}(\Gamma)<+\infty$. Then, for $\H^{N-1}$-a.e. $x\in\Gamma$ there exists $\nu\in\S^{N-1}$ such that, for any $0<\delta\leq 1$,
\[
\lim_{r\to 0^+} \frac{\H^{N-1}\big(\Gamma\cap C_{\delta r}(x,r)\big)}{\omega_{N-1}r^{N-1}}=1\,.
\]
In particular, the $(N-1)$-dimensional density of $x$ with respect to $\Gamma$ is $1$, and the set $\Gamma$ admits a tangent hyperplane at $x$, namely, $\{z,\, (z-x)\cdot \nu=0\}$.
\end{theorem}

\section{The main technical result, Theorem~\mref{nobound}\label{sec:ThA}}

This section is devoted to proving Theorem~\mref{nobound}. Since the construction is quite involved, we subdivide it into several simpler lemmas. Let $E$ and $\Omega$ be two sets of locally finite perimeter. Keeping in mind that $\nu_E(x)=\pm \nu_\Omega(x)$ for $\H^{N-1}$-a.e. $x\in \partial^* E\cap \partial^* \Omega$, we call $\Gamma^\pm_E$ the two sets given by
\begin{align}\label{defgpm}
\Gamma^+_E = \big\{ x\in \partial^* E \cap \partial^* \Omega: \nu_E(x)=\nu_\Omega(x) \big\}\,, && \Gamma^-_E = \big\{ x\in \partial^* E \cap \partial^* \Omega : \nu_E(x)=-\nu_\Omega(x) \big\}\,.
\end{align}
The simple meaning of these sets is evident in Figure~\ref{FigGpme}. In order to show Theorem~\mref{nobound}, we have to ``push'' both $\Gamma^+_E$ and $\Gamma^-_E$ outside of $\Omega$. This means that the set $E$ has to be somehow ``enlarged'' around $\Gamma^+_E$, and ``retracted'' around $\Gamma^-_E$. Things are even more complicated because we do not want to change $E$ inside $\Omega$.
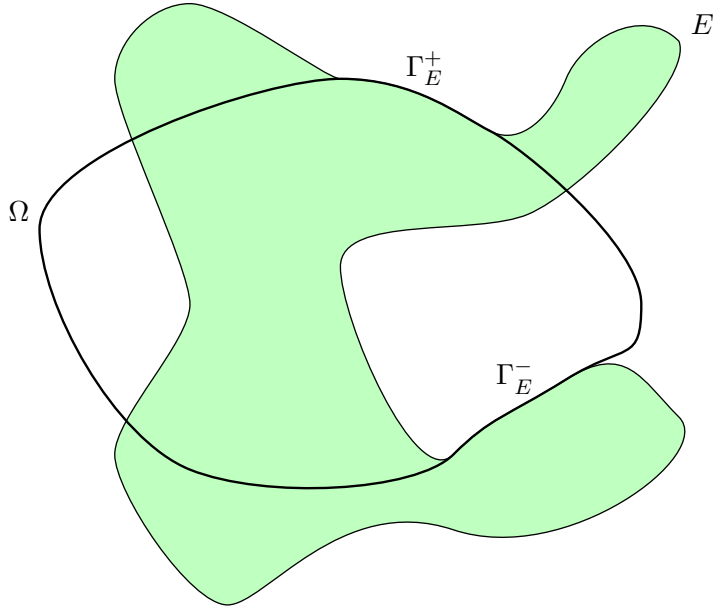
\begin{figure}[thbp]
\begin{tikzpicture}
\filldraw[fill=green!25, line width=.5pt] (3.5,-4) .. controls (2,-3.5) and (1,-5) .. (0.5,-5) .. controls (0,-5) and (-1,-3.5) .. (-1,-3) .. controls (-1,-2.5) and (0,-1.5) .. (0,-1) .. controls (0,-.5) and (-1,1.5) .. (-1,2) .. controls (-1,2.5) and (-.5,3) .. (0,3) .. controls (.5,3) and (1.8,2) .. (2,2) .. controls (3,2) and (3.6,1.5) .. (4,1.3) .. controls (4.4,1.1) and (4.8,1.5) .. (5,2) .. controls (5.2,2.5) and (6,3) .. (6.5,2.5) .. controls (6.7,2) and (5.2,0.5) .. (4.5,0.2) .. controls (3.8,-.1) and (2,0.2) .. (2,-.5) .. controls (2,-1.2) and (3,-3.5) .. (3.5,-3) .. controls (4,-2.5) and (4.2,-2.5) .. (5,-2) .. controls (5.8,-1.5) and (6,-2) .. (6.5,-2.5) .. controls (7,-3) and (5,-4.5) .. (3.5,-4);
\draw[line width=.9pt] (-2,0) .. controls (-2,1) and (1,2) .. (2,2) .. controls (3,2) and (3.6,1.5) .. (4,1.3) .. controls (4.4,1.1) and (6,-.2) .. (6,-1) .. controls (6,-1.8) and (5.8,-1.5) .. (5,.-2) .. controls (4.2,-2.5) and (4,-2.5) .. (3.5,-3) .. controls (3,-3.5) and (1,-3.6) .. (0,-3.2) .. controls (-1,-2.8) and (-2,-1) .. (-2,0);
\draw (-2,.5) node[anchor=north east] {$\Omega$};
\draw (7.1,3) node[anchor=north east] {$E$};
\draw (3.5,2.5) node[anchor=north east] {$\Gamma^+_E$};
\draw (4.7,-1.6) node[anchor=north east] {$\Gamma^-_E$};
\end{tikzpicture}
\caption{Definition of $\Gamma^\pm_E$.}
\label{FigGpme}
\end{figure}

Thanks to the blow-up properties of the perimeter, and keeping in mind Theorem~\ref{hnleb}, we easily get the following property.
\begin{lemma}\label{lemma0}
Let $x$ be a point of $(N-1)$-density $1$ for $\Gamma^+_E$, and call $\nu=\nu_E(x)$. Then, for every small $\delta>0$ there exists $\bar r>0$ such that for every $r<\bar r$ one has
\begin{gather}
|C^+(x,r)\cap (E\cup \Omega)| \leq \delta^2 \omega_{N-1}r^N\,, \label{uno} \\
|C^-(x,r)\setminus(E\cap \Omega)| \leq \delta^2 \omega_{N-1}r^N\,,\label{due}\\
\H^{N-1}\big(\Gamma^+_E\cap C_{\delta r}(x,r)\big) \geq (1-\delta)\omega_{N-1}r^{N-1}\,,\label{tre}\\
\H^{N-1}\big(\partial^* E \cap C(x,(1+\delta)r)\big) \leq (1+N\delta)\omega_{N-1}r^{N-1}\,,\label{treemezzo}\\
\H^{N-1}\big(\partial^*\Omega\cap C(x,(1+\delta)r)\big) \leq (1+N\delta) \omega_{N-1}r^{N-1}\,.\label{quattro}
\end{gather}
\end{lemma}
\begin{proof}
Let $\sigma>0$ be fixed. Since $x\in \Gamma^+$, and thus the vector $\nu$ is the outer normal at $x$ both to $E$ and to $\Omega$, for every $r$ small enough we have $|C^+(x,r)\cap E|< \sigma r^N$ and $|C^-(x,r)\setminus E|< \sigma r^N$, and the same holds with $\Omega$ in place of $E$. The validity of~(\ref{uno}) and~(\ref{due}) for $r\ll 1$ is then clear. The validity of~(\ref{tre}), (\ref{treemezzo}) and~(\ref{quattro}) for small $r$ is true because $x$ is a point of $(N-1)$-dimensional density $1$ for $\Gamma^+_E$, so also for $\partial^* E$, and $x\in \partial^* \Omega$, and of the fact that $(1+\delta)^{N-1} < 1 + N \delta$ for any small $\delta$.
\end{proof}

The first brick in our construction is the following result, which precisely explains how to ``push'' $\Gamma^+_E$ outside of $\Omega$ near a point $x\in\Gamma^+_E$.

\begin{lemma}\label{lemma1}
Let $x$ be a point of $(N-1)$-dimensional density $1$ for $\Gamma^+_E$, let $\delta>0$ be small, and let $r>0$ be such that~(\ref{uno})--(\ref{quattro}) hold. Then, there exists $h\in (\delta r,2 \delta r)$ such that, calling for brevity $\C=C_h(x,r)$, and setting $\widetilde E=E \cup (\C\setminus \Omega)$, one has
\begin{gather}
\widetilde E \supseteq E\,, \qquad \widetilde E\cap \Omega = E \cap \Omega,\, \qquad \big| \widetilde E \setminus E\big| \leq 4 \delta \omega_{N-1} r^N\,,\label{sei}\\
\H^{N-1}\Big(\partial^* E \setminus \partial^* \widetilde E \Big)\geq \H^{N-1}\Big(\big(\partial^* E \setminus \partial^* \widetilde E\big)\cap \Gamma^+_E \Big)\geq (1-\delta) \omega_{N-1} r^{N-1}\,, \label{sette}\\
\H^{N-1}\Big(\partial^* E \setminus \partial^* \widetilde E \Big)\leq (1+N\delta)\omega_{N-1}r^{N-1}\,, \label{ottoemezzo}\\
\H^{N-1}\Big(\partial^* \widetilde E \setminus \partial^* E \Big)\leq (1+5N\delta) \omega_{N-1}  r^{N-1}\,,\label{otto}\\
\H^{N-1}\Big(\big(\partial^* \widetilde E \setminus \partial^* E\big)\cap \partial^*\Omega \Big)\leq (N+1)\delta  \omega_{N-1}  r^{N-1} \,,\label{nove}\\
\H^{N-1}\Big(\partial^* \widetilde E \setminus \partial^* E \Big)\geq (1-3\delta)\omega_{N-1}  r^{N-1}\,.\label{nove+}
\end{gather}
Moreover, for $\H^{N-1}$-a.e. $z\in \big(\partial^* \widetilde E \setminus \partial^* E\big)\cap \partial^*\Omega$ we have that $\nu_{\widetilde E}(z) = -\nu_\Omega(z)$.
\end{lemma}

\begin{figure}[h!]
\begin{tikzpicture}[>=>>>]
\filldraw[green!25, line width=.5pt, draw=black] (2,1) .. controls (5,3) and (5,2.8) .. (6,3) .. controls (7,3.1) and (7,3.1) ..(8,3).. controls (8.2,3) and (8.6,2.8) .. (8.6,2.8)  .. controls (9,2.5) and (10,1.5) .. (11,1);
\draw[line width=.5pt] (6,2.2) .. controls (6,3) and (6,3.5) .. (6,4) .. controls (6.5,4) and (7.5,4) .. (8,4) .. controls (8,3.4) and (8,3.1) .. (8,2.2).. controls (7.5,2.2) and (6.8,2.2) .. (6,2.2);
\draw[line width=1.2pt] (2,1) .. controls (5,3) and (5,2.8) .. (6,3) .. controls (7,3.1) and (7,3.1) .. (8,3).. controls (8.2,3) and (8.6,2.8) .. (8.6,2.8)  ..  controls (8.8,2.7) and (8.9,2.6) .. (9,2.7) .. controls (9.1,2.8) and (9.1,2.9) .. (9.1,3) .. controls (9.1,3.1) and (8.9,3.2) .. (8.6,3.3)..  controls (8,3.5) and (7.6,3.8) .. (7.8,4)  ..  controls (9,5) and (10.5,4.5) .. (10,1);
\draw (9.5,3.85) node[anchor=north east] {$\Omega$};
\draw (7.75,1.75) node {$E$};
\draw (7,2.75) node{$x$};
\draw[->,line width=1pt] (7,3.1) -- (7,4.3);
\filldraw[black] (7,3.1) circle (1pt);
\draw (3,2.5) node[anchor=north east] {$\Gamma^+_E$};
\draw (6,4.3) node[anchor=north east] {$\mathcal{C}$};
\draw (7.5,4.9) node[anchor=north east] {$\nu_E(x)$};
\fill[pattern=north east lines] 
(2,1) .. controls (5,3) and (5,2.8) .. (6,3)
.. controls (6,3) and (6,3.5) .. (6,4) .. controls (6.5,4) and (6.5,4) .. (7,4).. controls (7,4) and (7,4) .. (7,3) .. controls (7,3) and (7,3) .. (6.75,3) .. controls (6.75,3) and (6.75,3) .. (6.75,2.5).. controls (6.75,2.5) and (6.75,2.5) .. (7.25,2.5).. controls (7.25,2.5) and (7.25,2.5) .. (7.25,3).. controls (7.25,3) and (7.25,3) .. (7,3).. controls (7,3) and (7,3) .. (7,4)..controls (7,4) and (7,4) ..(7.8,4) ..  controls  (7.6,3.8) and (8,3.5) .. (8,3.5)
 .. controls (8,3.4) and (8,3.1)
..(8,3).. controls (8.2,3) and (8.6,2.8) .. (8.6,2.8)  .. controls (9,2.5) and (10,1.5) .. (11,1) .. controls (10,1) and (9,1) .. (8,1).. controls (8,2) and (8,2) .. (8,2).. controls (8,2) and (8,2) .. (7.5,2).. controls (7.5,2) and (7.5,2) .. (7.5,1.5).. controls (7.5,1.5) and (7.5,1.5) .. (8,1.5).. controls (8,1.8) and (8,1.5) .. (8,1);
\end{tikzpicture}
\caption{The green set is $E$, the dashed set is $\widetilde E$, the thicker black line is $\partial^*\Omega$.}
\label{e_tilde}
\end{figure}
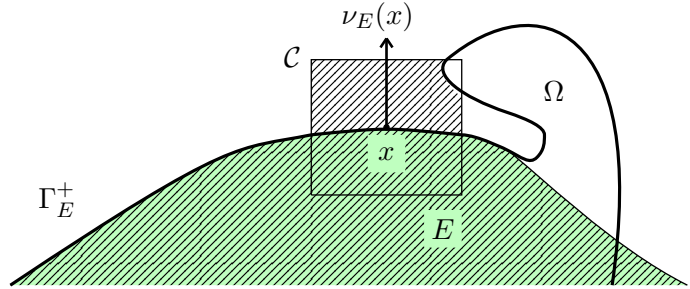

\begin{proof}
Let $x,\,\delta$ and $r$ be as in the statement. For every $t\in(-r,r)$, we let $D_t$ be the $(N-1)$-dimensional disk of radius $r$ centered at $x+t\nu_E(x)$ and orthogonal to $\nu_E(x)$. Moreover, let
\[
A_t := \left\{\begin{array}{cc}
D_t \setminus (E\cap \Omega) &\hbox{if $t<0$}\,, \\
D_t \cap (E\cup \Omega) &\hbox{if $t>0$}\,. 
\end{array}\right.
\]
We claim the existence of some $h\in (\delta r, 2 \delta r)$ such that
\begin{equation}\label{cinque}
\H^{N-1}(A_h \cup A_{-h}) \leq 3\delta \omega_{N-1} r^{N-1}\,.
\end{equation}
In fact, if the opposite inequality holds for almost all $h\in (\delta r, 2\delta r)$, then by Fubini Theorem and thanks to~(\ref{uno}) and~(\ref{due}) we would have
\[\begin{split}
2\delta^2\omega_{N-1}r^N
&\geq |C^+(x,r)\cap (E\cup \Omega)|+|C^-(x,r)\setminus(E\cap \Omega)| 
=\int_{-r}^r\H^{N-1}(A_t)\, dt \\
&\geq \int_{\delta r}^{2\delta r}\H^{N-1}(A_t\cup A_{-t})\, dt 
> 3\delta^2 \omega_{N-1} r^N\,,
\end{split}\]
which is a contradiction. Let then $h\in (\delta r, 2\delta r)$ be such that~(\ref{cinque}) holds, and call $\C=C_h(x,r)$ and $\widetilde E=E \cup (\C\setminus \Omega)$ as in the claim. We immediately notice that the validity of~(\ref{sei}) is obvious by construction.

Let us now consider a point $z\in \Gamma^+_E\cap C_{\delta r}(x,r)$. Up to a $\H^{N-1}$-negligible subset, any such point is contained in $\partial^* E$ since $\Gamma^+_E\subseteq \partial^* E$; moreover, it has density $1/2$ with respect to both $E$ and $\Omega$, and since $\nu_E(z)=\nu_\Omega(z)$ by definition of $\Gamma^+_E$, then it has density $1$ with respect to $\widetilde E$, using the fact that $h>\delta r$. Hence, $z\notin \partial^* \widetilde E$. Therefore, also by~(\ref{tre}) we get the validity of~(\ref{sette}).

The validity of~(\ref{ottoemezzo}) is clear by~(\ref{treemezzo}) and since $\H^{N-1}(\partial^* E\setminus \partial^* \widetilde E)\leq \H^{N-1}(\partial^* E \cap \overline{C(x,r)})$.

Let us then pass to consider a point $z\in \partial^* \widetilde E\setminus \partial^* E$, which must of course be contained in $\overline \C$. Again up to a $\H^{N-1}$-negligible subset, the point $z$ must have density $1/2$ with respect to $\widetilde E$ and density $0$ with respect to $E$ (keep in mind that $\widetilde E\supseteq E$). If $z\in\C$, we deduce that it has necessarily density $1/2$ with respect to $\Omega$; otherwise, it must belong to $\partial\C$. Using~(\ref{tre}) and~(\ref{quattro}) we immediately get that
\begin{equation}\label{ptl}
\H^{N-1}\Big(\big(\partial^*\Omega \setminus \partial^* E\big)\cap \overline \C\Big)\leq (N+1) \delta \omega_{N-1}r^{N-1}\,,
\end{equation}
which in particular implies~(\ref{nove}). To get~(\ref{otto}), we have to consider the points of $\big(\partial^*\widetilde E \setminus \partial^* E\big)\setminus \partial^* \Omega$. Since, as already observed, these points must be contained in $\partial\C$, we can write
\[
\big(\partial^*\widetilde E \setminus \partial^* E\big)\setminus \partial^* \Omega = Q_{\rm lat} \cup Q_{\rm down} \cup Q_{\rm up}\,,
\]
where $Q_{\rm lat}$ are the points in the lateral boundary of $\C$, $Q_{\rm down}$ those in the lower basis of $\C$, and $Q_{\rm up}$ those in the upper basis of $\C$. Concerning $Q_{\rm lat}$, since it is contained in the lateral boundary of $\C$, we simply estimate
\[
\H^{N-1}(Q_{\rm lat}) \leq (N-1)\omega_{N-1} r^{N-2} \cdot (2 h) \leq  4\delta(N-1)\omega_{N-1} r^{N-1}\,.
\]
Similarly, since $Q_{\rm up}$ is contained in the upper boundary of $\C$, we clearly have
\[
\H^{N-1}(Q_{\rm up}) \leq \omega_{N-1} r^{N-1}\,.
\]
Concerning $Q_{\rm down}$, by construction it can only contain points of the lower basis of $\C$ which have density $0$ with respect to $E$, so in particular points of $A_{-h}$. Hence by~(\ref{cinque}) we can estimate
\[
\H^{N-1}(Q_{\rm down}) \leq 3\delta \omega_{N-1} r^{N-1}\,.
\]
Putting the last three estimates together, we get that
\[
\H^{N-1}\Big(\big(\partial^*\widetilde E \setminus \partial^* E\big)\setminus \partial^* \Omega\Big) \leq \big( 1 + \delta (4N-1)\big)\omega_{N-1} r^{N-1}\,,
\]
which by~(\ref{ptl}) gives~(\ref{otto}). In a similar way, by construction $Q_{\rm up}$ contains all the points of the upper basis of $\C$ which are not in $A_h$, so
\[
\H^{N-1}\big(\partial^*\widetilde E \setminus \partial^* E \big) 
\geq \H^{N-1}(Q_{\rm up})\geq \omega_{N-1} r^{N-1} - \H^{N-1}(A_h)\,,
\]
which by~(\ref{cinque}) gives~(\ref{nove+}).

To conclude, let us consider a point $z\in (\partial^* \widetilde E\setminus \partial^* E)\cap \partial^* \Omega$. Up to a $\H^{N-1}$-negligible set, this point must have density $1/2$ with respect to $\widetilde E$ and to $\Omega$, and $0$ with respect to $E$. Since $\widetilde E\cap \Omega = E\cap \Omega$, the density of $z$ in $\widetilde E\cap\Omega$ is the same as in $E\cap \Omega$, that is, $0$. As a consequence, the outer normal vector $\nu_{\widetilde E}(z)$ to $\widetilde E$ at $z$ is $-\nu_\Omega(z)$. This concludes the proof.
\end{proof}

The above lemma allows to properly enlarge the set $E$ near a point $x\in\Gamma^+_E$; in order to deal with all points in $\Gamma^+_E$, we have to choose a sequence of points $x_j$ and perform a recursive construction. Let us be precise. We keep for the moment $\delta$ fixed; then, for every $x\in \Gamma^+_E$, we set
\begin{equation}\label{defr0}
r^0(x)= \max \Big\{ r \in [0, 1]: \hbox{properties~(\ref{uno})--(\ref{quattro}) hold} \Big\}\,.
\end{equation}
It is immediate from the definition that the above maximum actually exists. Moreover, Lemma~\ref{lemma0} implies that properties~(\ref{uno})--(\ref{quattro}) hold for all $r$ smaller than some suitable $\bar r$, so in particular $r^0(x)>0$; however, notice carefully that this does not imply that they hold for all $r$ smaller than $r^0(x)$. Unless $\Gamma^+_E=\emptyset$, we can then fix a point $x_1\in \Gamma^+_E$ which almost maximizes $r^0$, in the sense that
\[
r_1 := r^0(x_1) \geq \frac 12\, \sup \Big\{ r^0(x): x\in\Gamma^+_E\Big\}\,.
\]
We hence denote for brevity the cylinder provided by Lemma~\ref{lemma1} with the choice of $x=x_1$ and $r=r_1$ by  $\C_1$, and the set denoted as $\widetilde E$ in Lemma~\ref{lemma1} as $E_1=E\cup (\C_1\setminus \Omega)$. Now, as in~(\ref{defgpm}) we call $\Gamma^+_{E_1}$ the set of points in $\partial^* E_1\cap \partial^* \Omega$ for which $\nu_{E_1}=\nu_\Omega$. We observe that
\begin{equation}\label{noth}
\Gamma^+_{E_1} = \Gamma^+_E \cap \partial^* E_1\,.
\end{equation}
Indeed, one inclusion is true because if $x\in \Gamma^+_E \cap \partial^* E_1$ then $\nu_{E_1}(x)=\nu_E(x)$, since $E_1\supseteq E$. The other inclusion is given directly by Lemma~\ref{lemma1}, since a point in $(\partial^* E_1 \setminus \partial^* E)\cap \partial^*\Omega$ is not contained in $\Gamma^+_{E_1}$. Now, similarly to~(\ref{defr0}), for every $x\in \Gamma^+_{E_1}$ we set
\begin{equation}\label{defr1}
r^1(x)= \max \Big\{ r\in [0,1]: \hbox{(\ref{uno})--(\ref{quattro}) with $E_1$ and $\Gamma^+_{E_1}$ in place of $E$ and $\Gamma^+_E$ hold} \Big\}\,.
\end{equation}
Then, as before, unless $\Gamma^+_{E_1}$ is empty we can fix $x_2\in \Gamma^+_{E_1}$ almost maximizing $r^1$, i.e., so that
\[
r_2 := r^1(x_2) \geq \frac 12\, \sup \Big\{ r^1(x): x\in\Gamma^+_{E_1}\Big\}\,,
\]
and we call $\C_2$ the cylinder given by Lemma~\ref{lemma1} with the set $E_1$ in place of $E$, and with $x=x_2$ and $r=r_2$; moreover, we set $E_2=E_1\cup (\C_2\setminus \Omega)$. Going on with the obvious recursion, we get a sequence of points $\{x_j\}$, an increasing sequence of sets $\{E_j\}$, and a decreasing sequence of sets $\{\Gamma^+_{E_j}\}$, which as in~(\ref{noth}) satisfy
\begin{equation}\label{noth2}
\Gamma^+_{E_j} = \Gamma^+_{E_{j-1}} \cap \partial^* E_j= \Gamma^+_E \cap \partial^* E_j\,.
\end{equation}
Note that the above sequences can be actually finite if after finitely many steps the set $\Gamma^+_{E_j}$ remains empty. We can now prove the following property.
\begin{lemma}\label{lm2}
With the above notation, we have
\begin{equation}\label{lemma2}
\sum r_j^{N-1} \leq \frac 2{\omega_{N-1}}\, \H^{N-1}(\Gamma^+_E)\,.
\end{equation}
\end{lemma}
\begin{proof}
By~(\ref{noth2}) we have
\[
\Gamma^+_{E_{j-1}}\setminus \Gamma^+_{E_j} = \Gamma^+_{E_{j-1}}\setminus \partial^* E_j = \Gamma^+_{E_{j-1}}\cap \big( \partial^* E_{j-1}\setminus \partial^* E_j\big)\,.
\]
Hence, by~(\ref{sette}) of Lemma~\ref{lemma1} and since $\delta$ is small we have
\[
\H^{N-1}\big( \Gamma^+_{E_{j-1}}\setminus \Gamma^+_{E_j} \big) \geq (1-\delta) \omega_{N-1} r_j^{N-1} \geq \frac 12\, \omega_{N-1} r_j^{N-1}\,.
\]
The claim immediately follows since by construction the sequence $\Gamma^+_{E_j}$ is decreasing, and so
\[
\sum r_j^{N-1} \leq \frac 2{\omega_{N-1}}\, \sum \H^{N-1} \big(\Gamma^+_{E_{j-1}}\setminus \Gamma^+_{E_j} \big) 
\leq \frac 2{\omega_{N-1}}\,  \H^{N-1} \big(\Gamma^+_E \big) \,.
\]
\end{proof}

Since the sequence $\{E_j\}$ is increasing, we can define $E_\infty= \cup E_j$. We can show that the set $E_\infty$ is only slightly larger than $E$, and the corresponding set $\Gamma^+_{E_\infty}$ is empty.

\begin{lemma}\label{34}
With the above notation, we have that
\begin{gather}
E_\infty\supseteq E,\, \qquad E_\infty \cap \Omega = E\cap \Omega,\, \qquad |E_\infty\setminus E| \leq 8 \delta \H^{N-1}(\Gamma^+_E)\,, \label{volinf} \\
P(E_\infty) \leq P(E) + 2(1+5N) \delta \H^{N-1}(\Gamma^+_E) \,, \label{popepi} \\
P(E_\infty) \geq P(E) - 2(N+3) \delta\H^{N-1}(\Gamma^+_E)\,, \label{qual}\\
\H^{N-1}\big(\Gamma^+_{E_\infty}\big)=0\,. \label{no+}
\end{gather}
\end{lemma}
\begin{proof}
Since the sequence $\{E_j\}$ is increasing, the validity of~(\ref{volinf}) immediately comes by~(\ref{sei}) and~(\ref{lemma2}), keeping in mind that every $r_j$ is at most $1$.\par

Putting together~(\ref{sette}) and~(\ref{otto}), for every $j$ we have
\begin{equation}\label{ciaf}\begin{split}
P(E_j)-P(E_{j-1}) &= \H^{N-1} \big(\partial^* E_j \setminus \partial^* E_{j-1} \big)- \H^{N-1} \big(\partial^* E_{j-1} \setminus \partial^* E_j \big)\\
&\leq (1+5N\delta) \omega_{N-1}  r_j^{N-1} - (1-\delta) \omega_{N-1} r_j^{N-1}
= (1+5N)\delta \omega_{N-1}  r_j^{N-1}\,,
\end{split}\end{equation}
which implies
\[
P(E_j) \leq P(E) + (1+5N)\delta \omega_{N-1} \sum_{i=1}^j r_i^{N-1}\,.
\]
In a similar way, putting together~(\ref{ottoemezzo}) and~(\ref{nove+}), for every $j$ we have
\[\begin{split}
P(E_j)-P(E_{j-1}) &= \H^{N-1} \big(\partial^* E_j \setminus \partial^* E_{j-1} \big)- \H^{N-1} \big(\partial^* E_{j-1} \setminus \partial^* E_j \big)\\
&\geq (1-3\delta)\omega_{N-1}  r_j^{N-1}-(1+N\delta)\omega_{N-1}r_j^{N-1}
= -(N+3)\delta \omega_{N-1} r_j^{N-1}\,,
\end{split}\]
which implies
\[
P(E_j) \geq P(E) - (N+3)\delta \omega_{N-1} \sum_{i=1}^j r_i^{N-1}\,.
\]
We can now notice that $P(E_j)\to P(E_\infty)$ by the general Proposition~\ref{genprop} below, which can be applied in our case thanks to~(\ref{otto}) and Lemma~\ref{lm2}. As a consequence, the last two estimates together with~(\ref{lemma2}) imply~(\ref{popepi}) and~(\ref{qual}).\par

Let us now show that
\begin{equation}\label{onth}
\H^{N-1} \big( \Gamma^+_{E_\infty} \setminus \Gamma^+_E\big)=0\,.
\end{equation}
Indeed, let $x$ be a point of $(N-1)$-dimensional density $1$ with respect to $\Gamma^+_{E_\infty}\setminus \Gamma^+_E$. Then, $x$ has density $1/2$ with respect to $E_\infty$; since $E_\infty\supseteq E$, if $x$ has density $1/2$ also with respect to $E$ then the outer normal vector is the same, and then $x\in \Gamma^+_{E_\infty}\cap \Gamma^+_E$, against the assumption. As a consequence, the density of $x$ with respect to $E$ is not $1/2$, thus it is $0$ because $E\subseteq E_\infty$. But then, since $E_\infty=E$ inside $\Omega$, the density of $x$ with respect to $E_\infty\cap\Omega$ is $0$, and since $x\in\partial^*\Omega$ this implies that $\nu_{E_\infty}(x)=-\nu_\Omega(x)$, against the fact that $x\in \Gamma^+_{E_\infty}$. The property~(\ref{onth}) is then established.

As a consequence, to show~(\ref{no+}) we take a point $x$ of $(N-1)$-dimensional density $1$ for $\Gamma^+_E \cap \Gamma^+_{E_\infty}$, and we seek a contradiction. We fix a small $r$ such that
\begin{gather}
|C^+(x,r)\cap (E_\infty \cup \Omega)|\leq \delta^2\omega_{N-1}r^N\,, \label{uno'} \\
|C^-(x,r)\setminus(E_\infty \cap \Omega)|\leq \frac{\delta^2} 2 \, \omega_{N-1}r^N\,,\label{due'}\\
\H^{N-1}\big(\Gamma^+_{E_\infty} \cap \Gamma^+_E\cap C_{\delta r}(x,r)\big)\geq (1-\delta)\omega_{N-1}r^{N-1}\,,\label{tre'}\\
\H^{N-1}\big(\partial^* E_\infty \cap C(x,(1+\delta) r)\big) \leq \big(1+(N-1/2)\delta\big)\omega_{N-1}r^{N-1}\,,\label{treemezzo'}\\
\H^{N-1}\big(\partial^*\Omega\cap C(x,(1+\delta)r)\big)\leq (1+N\delta) \omega_{N-1}r^{N-1}\,.\label{quattro'}
\end{gather}
Notice that these properties are almost identical to~(\ref{uno})--(\ref{quattro}), the only difference being that the set $E$ is replaced by $E_\infty$ everywhere, that $\Gamma^+_E$ in~(\ref{tre}) becomes $\Gamma^+_{E_\infty}\cap\Gamma^+_E$ in~(\ref{tre'}), and that $\delta^2$ in~(\ref{due}) has been replaced by $\delta^2/2$ in~(\ref{due'}), as well as $N\delta$ in~(\ref{treemezzo}) has been replaced by $(N-1/2)\delta$ in~(\ref{treemezzo'}). The existence of such a $r$ can be proved exactly as in Lemma~\ref{lemma0}, since $x$ is a point of $(N-1)$-dimensional density $1$ for $\Gamma^+_E \cap \Gamma^+_{E_\infty}$, hence also for $\partial^* E_\infty$ (and replacing $N$ with $N-1/2$ is not a problem because also $(1+\delta)^{N-1}<1+(N-1/2)\delta$ is true, since $\delta$ is small). Now, notice that $x$ has density $1/2$ with respect to both $E$ and $E_\infty$, thus also with any intermediate $E_j$, and in particular $x\in \Gamma^+_{E_j}$ for every $j$. Take now an index $j$; since $E_\infty\supseteq E_j$, the validity of~(\ref{uno'}) with $r$ implies that of~(\ref{uno}), \emph{with the same value of $r$}, with $E_j$ in place of $E$. Analogously, $\Gamma^+_{E_\infty}\cap \Gamma^+_E\subseteq \Gamma^+_{E_j}$, because every point in $\Gamma^+_{E_\infty}\cap \Gamma^+_E$ has density $1/2$ with respect to both $E_\infty$ and $E$, thus also with respect to any $E_j$, hence it is contained in every $\Gamma^+_{E_j}$; as a consequence, the validity of~(\ref{tre'}) with $r$ implies the validity, with the same $r$, of~(\ref{tre}) with $\Gamma^+_{E_j}$ in place of $\Gamma^+_E$. Property~(\ref{quattro'}) is exactly the same as~(\ref{quattro}). The property~(\ref{due}) for $E_j$ does not immediately follow from~(\ref{due'}), because $E_\infty\supseteq E_j$ gives that $|C^-(x,r)\setminus(E_\infty \cap \Omega)|$ is smaller, and not larger, than $|C^-(x,r)\setminus(E_j \cap \Omega)|$. Nevertheless, the presence of $\delta^2/2$ in place of $\delta^2$ actually allows to infer~(\ref{due}) for $E_j$ from~(\ref{due'}), again with the same value of $r$, as soon as
\[
|E_\infty \setminus E_j |\leq \frac{\delta^2} 2\, \omega_{N-1} r^N\,,
\]
and this is surely true for every $j$ large enough since as already noticed $\Chi{E_j}\to \Chi{E_\infty}$ in $L^1$. It remains to deduce~(\ref{treemezzo}) with $\partial^* E_j$ in place of $\partial^* E$ from~(\ref{treemezzo'}). However, we know that $\Chi{E_j\cap C(x,r)}\to \Chi{E_\infty\cap C(x,r)}$, and
\[
\sum_{j\in\N} \H^{N-1}(\partial^* E_{j+1}\setminus \partial^* E_j) <+\infty\,,
\]
which is clear because the estimate~(\ref{ciaf}) holds and $\sum_j r_j^{N-1}$ converges. Therefore, we can apply Proposition~\ref{genprop} to obtain that $\partial^* E_\infty = \cap_{k\in\N} \cup_{j\geq k} \partial^* E_j$. As a consequence, for every $j$ large enough we have
\[
\H^{N-1}(\partial^* E_j \setminus \partial^* E_\infty) \leq \frac\delta 2\,\omega_{N-1}r^{N-1}
\]
and then for those $j$ in fact~(\ref{treemezzo}) with $\partial^* E_j$ in place of $\partial^* E$ holds thanks to~(\ref{treemezzo'}).\par

Summarizing, the validity of~(\ref{uno'})--(\ref{quattro'}) implies the validity of~(\ref{uno})--(\ref{quattro}) with $E_j$ and $\Gamma^+_{E_j}$ in place of $E$ and $\Gamma^+_E$ for every $j$ large enough. Since the values of $r^j$ have been defined via the analogue of~(\ref{defr1}), this implies that $r^j(x)\geq r$ for every $j\gg 1$. And in turn, since by definition for every $j$ we have $r_{j+1}\geq r^j(x)/2$, we obtain $r_j\geq r/2$ for every $j\gg 1$. Finally, this is impossible since $r>0$ while $r_j \to 0$ by~(\ref{lemma2}), yielding the desired contradiction.
\end{proof}

Summarizing, by now we know that the set $E_\infty$, which is only slightly larger than $E$, does not have points in $\Gamma^+_{E_\infty}$; in other words, $\H^{N-1}$-a.e.\ point of $\partial^* E_\infty\cap \partial^* \Omega$ satisfies $\nu_{E_\infty}=-\nu_\Omega$, that is, $\partial^* E_\infty\cap\partial^* \Omega=\Gamma^-_{E_\infty}$. We want now to understand how large this latter set can be. The fact that $E_\infty\supseteq E$ but $E_\infty\cap\Omega = E\cap \Omega$ implies that $\Gamma^-_{E_\infty}\supseteq \Gamma^-_E$, and we can easily control $\H^{N-1}(\Gamma^-_{E_j}\setminus \Gamma^-_E)$ thanks to~(\ref{nove}). However, while it is clear that $\Gamma^-_{E_\infty}$ contains the increasing union of the sets $\Gamma^-_{E_j}$, in principle there is the risk that it is strictly larger. We will exclude this by means of Corollary~\ref{corlapi}, which is a consequence of the following general result.

\begin{prop}\label{genprop}
Let $\{F_j\}$ be a sequence of sets of finite perimeter such that $\Chi{F_j} \to \Chi F$ in $L^1$ and $\sum_j \H^{N-1}(\partial^* F_{j+1}\setminus \partial^* F_j)<+\infty$. Then $\partial^* F =\cap_{k\in\N} \cup_{j\geq k} \partial^* F_j$ up to $\H^{N-1}$-negligible sets. In particular, $P(F)=\lim_{j \to +\infty} P(F_j)$.
\end{prop}
\begin{proof}
First of all, we notice that
\[
P(F_h) - P(F_1) = \sum_{j=1}^{h-1} \H^{N-1}(\partial^* F_{j+1}\setminus \partial^* F_j) - \H^{N-1}(\partial^* F_j\setminus \partial^* F_{j+1})\,,
\]
thus by assumption we deduce that $\sum_j \H^{N-1}(\partial^* F_{j+1}\setminus \partial^* F_j)<+\infty$, so we actually have
\begin{equation}\label{sercon}
\sum_{j\in\N} \H^{N-1}(\partial^* F_{j+1}\triangle \partial^* F_j)<+\infty\,.
\end{equation}
Observe that the above formula for $P(F_h)$ in particular implies that the perimeters of the sets $F_h$ are uniformly bounded, so by lower semicontinuity also $F$ is a set of finite perimeter. Now, let us denote for brevity $\Gamma_k = \cup_{j\geq k} \partial^* F_j$, and $\Gamma = \cap_{k\in\N} \Gamma_k$, so the first part of the claim is rewritten as $\partial^* F = \Gamma$. Notice that $\H^{N-1}(\Gamma)<+\infty$ by~(\ref{sercon}), and actually $\H^{N-1}(\Gamma\triangle \partial^* F_j)$ becomes arbitrarily small if $j$ is large; thus, $\H^{N-1}(\Gamma)=\lim_{j\to+\infty} P(F_j)$, so the whole claim follows once we prove that $\partial^* F =\Gamma$. We split the proof into two parts, corresponding to the two inclusions.

\step{I}{One has $\partial^* F \subseteq \Gamma$.}
We just need to show that $\partial^* F \subseteq \cup_j \partial^* F_j$, which can be rewritten as $\partial^* F\subseteq \Gamma_1$. Indeed, this will automatically give also that $\partial^* F\subseteq \Gamma_k$ for every $k$, since we can apply the result to the sequence which starts with $F_k$, for which the assumptions are obviously verified, and then we have also $\partial^* F \subseteq \cap_k \Gamma_k$, which is the thesis of this step.\par

Let us then assume by contradiction that $\H^{N-1}(\partial^* F \setminus \cup_j \partial^* F_j)>0$, and according to Theorem~\ref{hnleb} let $x$ be a point for which a tangent hyperplane, with normal vector $\nu\in\S^{N-1}$, exists. In particular, of course $\nu=\nu_F(x)$. Calling as usual $C(x,r)$ the cylinder centered at $x$, with axis parallel to $\nu_F(x)$, radius $r$ and height $2r$, for every $z\in C(x,r)$ we write $z=z' + z_n$, with $z_n$ parallel to $\nu_F(x)$ and $z'$ in the $(N-1)$-dimensional disk $D$ with radius $r$, centered at $x$ and perpendicular to $\nu_F(x)$. For brevity, for every set $G$ of finite perimeter and for every $z'\in D$ we call $G_{z'}$ the $1$-dimensional section given by the intersection of $G$ with the segment $S_{z'}=\big\{z' + t \nu_F(x): t\in (-r,r)\big\}$. Moreover, we call $\partial^* G_{z'}=\partial^* G \cap S_{z'}$, which by Vol'pert Theorem coincides with the boundary of the one-dimensional section $G_{z'}$ in $C(x,r)$ for $\H^{N-1}$-a.e. $z'\in D$. Let us now call
\[
H = \bigg\{z' \in D: \H^1(F_{z'})\in \Big(\,\frac r2, \frac 32\, r\,\Big),\, \# \big(\partial^* F_{z'}\big) = \#\big( \partial^* F_{z'}\setminus \cup_j \partial^* F_{j,z'}\big) = 1 \bigg\}\,.
\]
Since $x$ is a point of $(N-1)$-dimensional density $1$ for $\H^{N-1}\big(\partial^* F \setminus \cup_j \partial^* F_j\big)$ and the outer normal vector to $F$ at $x$ is $\nu_F(x)$, by the blow-up properties of the perimeter we have that $\H^{N-1}(H)/r^{N-1}$ can be taken arbitrarily close to $\omega_{N-1}$ if $r$ is small enough. We can then take $r$ so small that
\begin{equation}\label{tq}
\H^{N-1}(H) \geq \frac 23\, \omega_{N-1} r^{N-1}\,.
\end{equation}
Moreover, by assumption we have some $\ell$ such that
\begin{equation}\label{fordom0}
\sum_{j=\ell}^{+\infty} \H^{N-1}(\partial^* F_{j+1}\setminus \partial^* F_j) < \frac{\omega_{N-1}}3\,  r^{N-1}\,.
\end{equation}
Consider now $z'\in H$. By definition, the section $F_{z'}$ is a segment with length between $r/2$ and $3r/2$, and with one endpoint in $C(x,r)$. Since this endpoint does not belong to $\partial^* F_{\ell,z'}$, neither the section $F_{\ell,z'}$ nor its complement $S_{z'}\setminus F_{\ell,z'}$ can coincide with $F_{z'}$, so in particular
\[
v(z')  = \min\Big\{\H^1(F_{\ell,z'}\triangle F_{z'}) ,\, \H^1\big((S_{z'}\setminus F_{\ell,z'})\triangle F_{z'}\big) ,\, \frac r2\,\Big\} >0\,.
\]
Notice now that, for every $j>\ell$,
\begin{equation}\label{fordom}
\H^1(F_{j,z'}\triangle F_{z'})< v(z') \quad \Longrightarrow \quad \partial^* F_{j,z'}\setminus \partial^* F_{\ell,z'}\neq \emptyset\,.
\end{equation}
Indeed, by definition of $v(z')$, if $\H^1(F_{j,z'}\triangle F_{z'})< v(z')$ then the section $F_{j,z'}$ coincides neither with $F_{\ell,z'}$ nor with its complement $S_{z'}\setminus F_{\ell,z'}$. As a consequence, $\partial^* F_{j,z'}$ must contain some point which is not in $\partial^* F_{\ell, z'}$ unless the section $F_{j,z'}$ is either empty or the whole segment $S_{z'}$. And in turn, we can exclude this possibility since by definition $v(z')\leq r/2$ and the length of $F_{z'}$ is between $r/2$ and $3r/2$. Thus, (\ref{fordom}) has been proved.\par

For any $j>\ell$, we now set $H_j = \{z' \in H: \H^1(F_{j,z'}\triangle F_{z'})<v(z')\}$, and by~(\ref{fordom}) and by projection we know that
\begin{equation}\label{fordom2}
\H^{N-1}\big(\partial^* F_j\setminus \partial^* F_\ell\big) \geq \H^{N-1}(H_j)\,.
\end{equation}
Since by Fubini Theorem we have
\[
\big|F \triangle F_j \big| \geq \int_{H\setminus H_j} v(z') \,d\H^{N-1}(z')\,,
\]
recalling that $\Chi{F_j}\to\Chi{F}$ and that $v>0$ we deduce that $\H^{N-1}(H\setminus H_j)$ becomes arbitrarily small for $j$ large enough, and by~(\ref{tq}) this implies the existence of some large $h>\ell$ such that
\[
\H^{N-1}(H_h) \geq \frac{\omega_{N-1}}3\, r^{N-1}\,.
\]
Therefore, recalling~(\ref{fordom2}) and~(\ref{fordom0}), we obtain
\[
\frac{\omega_{N-1}}3\, r^{N-1} \leq \H^{N-1}(H_h)\leq \H^{N-1}\big(\partial^* F_h\setminus \partial^* F_\ell\big)
\leq \sum_{j=\ell}^{h-1} \H^{N-1}\big(\partial^* F_{j+1}\setminus \partial^* F_j\big)< \frac{\omega_{N-1}}3\,  r^{N-1}\,,
\]
so we have obtained the desired contradiction and this step is concluded.

\step{II}{One has $\partial^* F \supseteq \Gamma$.}
Let us now concentrate on the opposite inclusion; we assume that $\H^{N-1}(\Gamma\setminus \partial^* F)>0$ and we look for a contradiction. According to Theorem~\ref{hnleb}, we can take a point $x$ of $(N-1)$-dimensional density $1$ both for $\Gamma$ and for $\Gamma\setminus \partial^* F$, with tangent hyperplane orthogonal to a vector $\nu\in\S^{N-1}$. As in the first step, for a given $r>0$ we call $D$ the $(N-1)$-dimensional disk centered at $x$, with radius $r$ and perpendicular to $\nu$, and for every $z'\in D$ we call $S_{z'}$ the segment of length $2r$ centered at $z'$ and parallel to $\nu$. This time, we define the set
\[
H = \bigg\{ z'\in D: \#(\Gamma\cap S_{z'})=\#(\Gamma\cap S_{z'}\cap C_{r/2}(x,r))=1,\, \partial^* F\cap S_{z'} = \emptyset\bigg\}\,.
\]
In words, the points $z'\in D$ are those for which in the segment $S_{z'}$ there are no points of $\partial^* F$, while there is exactly one point of $\Gamma$, and this point has distance at least $r/2$ from both ends of $S_{z'}$. Since $x$ has $(N-1)$-dimensional density $1$ with respect to both $\Gamma$ and $\Gamma\setminus\partial^* F$, and thus $0$ with respect to $\partial^* F$, the set $H$ is an arbitrarily high portion of $D$ if $r$ is small, and so in particular we can fix a small $r$ such that
\[
\H^{N-1}(H) \geq \frac 34\, \omega_{N-1}\, r^{N-1}\,.
\]
Having fixed $r$, by~(\ref{sercon}) we can now take $k$ large enough so that
\begin{align}\label{thiseasy}
\sum_{j\geq k} \H^{N-1}(\partial^* F_{j+1}\triangle \partial^* F_j)< \frac{\omega_{N-1}}4\, r^{N-1}\,, && |F \triangle F_k| < \frac{\omega_{N-1}}{10} \, r^N\,.
\end{align}
The first estimate in particular gives $\H^{N-1}(\Gamma\triangle \partial^* F_k)<\omega_{N-1} r^{N-1}/4$, so that calling
\[
H' = \Big\{ z'\in H: (\Gamma\triangle\partial^* F_k)_{z'} = \emptyset\Big\}
\]
we have
\[
\H^{N-1}(H') \geq \frac {\omega_{N-1}}2\, r^{N-1}\,.
\]
Let us now consider $z'\in H$. In the section $S_{z'}$ there is exactly one point of $\partial^* F_k$, and this point has distance at least $r/2$ from the endpoints of $S_{z'}$. Hence, $F_k\cap S_{z'}$ is a segment with length between $r/2$ and $3r/2$. Instead, the section $S_{z'}$ does not intersect $\partial^* F$, thus $F\cap S_{z'}$ is either empty or the whole segment $S_{z'}$; in both cases, we have $\H^1(F \triangle F_k)\cap S_{z'})>r/2$. Therefore, by Fubini Theorem we have
\[
|F \triangle F_k| \geq \int_{z'\in H'} \H^1\big(F\triangle F_k)\cap S_{z'}\big)\, d\H^{N-1}(z') > \frac r2\, \H^{N-1}(H') \geq \frac {\omega_{N-1}}4\, r^N\,,
\]
and this contradicts the second estimate in~(\ref{thiseasy}). The desired inclusion is then proved, so also the second step of the proof is concluded.
\end{proof}

\begin{corollary}
\label{corlapi}
One has $(\partial^* E_\infty\setminus \partial^* E)\cap \partial^* \Omega=\cup_j (\partial^* E_j\setminus \partial^* E)\cap \partial^* \Omega$. In particular,
\begin{equation}\label{lapi}
\H^{N-1} \Big((\partial^* E_\infty\setminus \partial^* E)\cap \partial^* \Omega \Big)\leq 2 (N+1)\delta \H^{N-1}(\Gamma^+_E) \,.
\end{equation}
\end{corollary}
\begin{proof}
We can apply Proposition~\ref{genprop} to the sets $E_j$, since by~(\ref{otto})
\[
\H^{N-1}(\partial^* E_{j+1}\setminus \partial^* E_j)\leq (1+5N\delta) \omega_{N-1}  r_j^{N-1}
\]
and thanks to Lemma~\ref{lm2}. We get then that $\partial^* E_\infty\subseteq \cup_j \partial^* E_j$, thus of course
\[
(\partial^* E_\infty\setminus\partial^* E)\cap \partial^*\Omega\subseteq \cup_j (\partial^* E_j\setminus \partial^* E)\cap \partial^*\Omega\,.
\]
To prove the opposite inclusion, taken a point $z$ contained in some $(\partial^* E_j\setminus \partial^* E)\cap \partial^*\Omega$ we only have to show that $z\in \partial^* E_\infty$. Since $z\in(\partial^* E_j\setminus \partial^* E)\cap \partial^*\Omega$, Lemma~\ref{lemma1} ensure that $\nu_{E_j}(z) = - \nu_\Omega(z)$. Thus, the point $z$ has density $1/2$ with respect to $E_j$ and density $1/2$ with respect to $\Omega\setminus E_j$; since $E_\infty\supseteq E_j$, we deduce that $z$ has density  at least $1/2$ with respect to $E_\infty$, and since $E_\infty\cap \Omega= E_j\cap\Omega$, we deduce that $z$ has density $1/2$ with respect to $\Omega\setminus E_\infty$. This implies that the density of $z$ with respect to $E_\infty$ is actually $1/2$, so $z\in \partial^* E_\infty$ and this proves the desired inclusion.

Therefore, thanks to~(\ref{nove}) and~(\ref{lemma2}) we have
\[\begin{split}
\H^{N-1} \Big((\partial^* E_\infty\setminus \partial^* E)&\cap \partial^* \Omega \Big)
= \H^{N-1} \Big(\bigcup_j (\partial^* E_j\setminus \partial^* E)\cap \partial^* \Omega \Big)\\
&\leq \sum_j \H^{N-1} \big( (\partial^* E_{j+1}\setminus \partial^* E_j)\cap \partial^* \Omega\big)
\leq (N+1)\delta  \omega_{N-1} \sum_j   r_j^{N-1}\\
&\leq 2 (N+1)\delta \H^{N-1}(\Gamma^+_E)\,,
\end{split}\]
which gives~(\ref{lapi}) and concludes the proof.
\end{proof}

Collecting all the above results, we have then the following claim.

\begin{prop}\label{finally}
Let $\Omega$ and $E$ be two sets of finite perimeter. Then, for every $\eps>0$ there exists a set $E'$ such that
\begin{gather}
E'\supseteq E,\, \qquad E'\cap \Omega = E\cap\Omega,\, \qquad |E'\setminus E| < \eps\,, \label{tredici}\\
\H^{N-1}(\Gamma^+_{E'}) = 0\,, \qquad \H^{N-1}(\Gamma^-_{E'}\setminus \Gamma^-_E) < \eps \H^{N-1}(\Gamma^+_E)\,, \label{quattordici}\\
|P(E') - P(E)| \leq \eps\,, \label{quindici}\\
\H^{N-1}(\partial^* E' \setminus \partial^* E) < \H^{N-1}(\Gamma^+_E)+\eps\,. \label{sedici}
\end{gather}
\end{prop}
\begin{proof}
We fix
\[
\delta = \min \bigg\{\frac \eps{12 N  P(E)},\, \frac \eps{2(N+1)},\, \frac\eps{ 2(1+5N) P(E)},\, \frac 12\bigg\} \,,
\]
we perform the construction of the previous pages, and we set $E'=E_\infty$. We have to check that the properties~(\ref{tredici})--(\ref{sedici}) are satisfied.

The property~(\ref{tredici}) is true by construction and by~(\ref{volinf}), keeping in mind that $\Gamma^+_E\subseteq \partial^* E$.\par

The property~(\ref{quattordici}) is given by~(\ref{no+}), thanks to~(\ref{lapi}) and since by construction every point of $\Gamma^-_{E'}\setminus \Gamma^-_E$ must be contained in $\partial^* E_\infty\cap \partial^* \Omega$ and not in $\partial^* E$.\par

The property~(\ref{quindici}) is given by~(\ref{popepi}) and~(\ref{qual}).\par

Finally, to obtain the property~(\ref{sedici}), we first recall that $\partial^* E' \setminus \partial^* E\subseteq \cup_j \partial^* E_{j+1}\setminus \partial^* E_j$ by Proposition~\ref{genprop}, and then we use~(\ref{otto}) and~(\ref{sette}) and we keep in mind Lemma~\ref{lemma1} and the fact that the sets $\Gamma^+_{E_j}$ are decreasing to get that
\[\begin{split}
\H^{N-1}(\partial^* E' \setminus \partial^* E) &\leq \sum_j \H^{N-1}(\partial^* E_{j+1} \setminus \partial^* E_j)\\
&\leq \frac{1+5N\delta}{1-\delta} \sum_j  \H^{N-1}\big((\partial^* E_j \setminus \partial^* E_{j+1})\cap \Gamma^+_{E_j}\big)\\
&= \frac{1+5N\delta}{1-\delta} \sum_j  \H^{N-1}\big(\Gamma^+_{E_j} \setminus \Gamma^+_{E_{j+1}}\big)
\leq \frac{1+5N\delta}{1-\delta}\, \H^{N-1}(\Gamma^+_E)\,,
\end{split}\]
which gives~(\ref{sedici}) by the definition of $\delta$.
\end{proof}

The above proposition finally gives a precise definition of how to enlarge $E$ so to eliminate $\Gamma^+_E$; looking at Figure~\ref{FigGpme} to get an idea, we are basically ``pushing $\Gamma^+_E$ outside $E$'', and so outside $\Omega$. We need also to eliminate $\Gamma^-_E$ by ``pulling $\Gamma^-_E$ inside $E$'', so once again outside $\Omega$. Since Proposition~\ref{finally} only requires $E$ and $\Omega$ to be of finite perimeter, but not of finite volume, we can apply it to $\Omega$ and $\R^N\setminus E$. Hence, we get for free the following analogous of Proposition~\ref{finally}.
\begin{prop}\label{finally2}
Let $\Omega$ and $E$ be two sets of finite perimeter. Then, for every $\eps>0$ there exists a set $E'$ such that
\begin{gather}
E'\subseteq E,\, \qquad E'\cap \Omega = E\cap\Omega,\, \qquad |E\setminus E'| < \eps\,, \tag{\ref{tredici}'}\\
\H^{N-1}(\Gamma^-_{E'}) = 0\,, \qquad \H^{N-1}(\Gamma^+_{E'}\setminus \Gamma^+_E) < \eps \H^{N-1}(\Gamma^-_E)\,, \tag{\ref{quattordici}'}\\
|P(E')- P(E)|\leq \eps\,, \tag{\ref{quindici}'}\\
\H^{N-1}(\partial^* E' \setminus \partial^* E) < \H^{N-1}(\Gamma^-_E)+\eps\,. \tag{\ref{sedici}'}
\end{gather}
\end{prop}

\begin{remark}
\label{rem:inside}
It is important to observe that one cannot get Theorem~\mref{nobound} just applying once Proposition~\ref{finally} and once Proposition~\ref{finally2}. In fact, applying Proposition~\ref{finally} to $E$ one gets some set $E'$ with $\Gamma^+_{E'}=\emptyset$ and with $\Gamma^-_{E'}$ only slightly larger than $\Gamma^-_E$. Applying then Proposition~\ref{finally2} to $E'$, one gets a set $E''$ with $\Gamma^-_{E''}=\emptyset$. However, even though $\Gamma^+_{E'}=\emptyset$, the set $\Gamma^+_{E''}$ is slightly larger and then it could be nonempty. As a consequence, the proof of Theorem~\mref{nobound} needs another recursive procedure. We underline that the situation is so complicated because in Theorem~\mref{nobound} we want $F\cap\Omega=E\cap\Omega$. If we did not want this extra property, then things would be a bit simpler. In fact, in Lemma~\ref{lemma1} one could simply define $\widetilde E=E\cup \C$ instead of $\widetilde E = E \cup (\C\setminus \Omega)$; all the results of this section would then remain true, with the very same proof or even with a simpler one. And in that case, the final result would actually follow by applying the enlargement once to eliminate $\Gamma^+$ and once to eliminate $\Gamma^-$.
\end{remark}

\proofof{Theorem~\mref{nobound}}
We prove the result under the assumption that both $E$ and $\Omega$ are of finite perimeter, since the case when they are just of locally finite perimeter then follows via a standard argument. Let then $E$ and $\Omega$ have finite perimeter, and fix $\eps>0$. We apply Proposition~\ref{finally} with constant $\eps/2$ to $E$ getting a set $E_1$; then, we apply Proposition~\ref{finally2} with constant $\eps/4$ to $E_1$ getting a set $E_2$; recursively, we go on defining each set $E_j$ by applying Proposition~\ref{finally} if $j$ is odd and Proposition~\ref{finally2} if $j$ is even, in both cases with constant $\eps/2^j$. Since $|E_j\triangle E_{j-1}|\leq \eps 2^{-j}$ for every $j$, there is a limiting set $F$ such that $\Chi{E_j}\to \Chi F$ in $L^1$. To conclude the proof we must check that this set satisfies~(\ref{eq:A.1})--(\ref{eq:A.4}).\par

The fact that $F\cap \Omega=E\cap \Omega$ is true by construction, hence~(\ref{eq:A.1}) follows. The estimate~(\ref{eq:A.2}) for the volume is true by construction and by~(\ref{tredici}) and~(\ref{tredici}').\par

By~(\ref{quattordici}) we have that
\begin{align*}
\H^{N-1}(\Gamma^+_{E_1})=0\,, &&
\H^{N-1}(\Gamma^-_{E_1})\leq \H^{N-1}(\Gamma^-_E) + \frac \eps 2 \H^{N-1}(\Gamma^+_E)\leq P(E)\,.
\end{align*}
Moreover, by~(\ref{quattordici}') we have that
\begin{align*}
\H^{N-1}(\Gamma^-_{E_2})=0\,, &&
\H^{N-1}(\Gamma^+_{E_2})\leq \H^{N-1}(\Gamma^+_{E_1}) + \frac \eps 4\, \H^{N-1}(\Gamma^-_{E_1})
\leq \frac \eps 4\, P(E)\,.
\end{align*}
Recursively, we obtain that for every $j\geq 2$
\[
\H^{N-1}(\partial^* E_j \cap \partial^* \Omega) =
\H^{N-1}\big(\Gamma^+_{E_j} \cup \Gamma^-_{E_j}\big) \leq \frac \eps {2^j}\, P(E)\,.
\]
Now, since by Proposition~\ref{genprop} we know that $\partial^* F = \cap_k \cup_{j\geq k} \partial^* E_j$, then
\[
\H^{N-1}\big(\partial^* F \cap \partial^* \Omega\big) = \lim_{j\to \infty} \H^{N-1}(\partial^* E_j \cap \partial^* \Omega) = 0\,,
\]
thus also~(\ref{eq:A.4}) is proved.\par

The above estimates, together with~(\ref{sedici}) and~(\ref{sedici}') imply in particular that
\[
\sum_j \H^{N-1}(\partial^* E_{j+1} \setminus \partial^* E_j) \leq \sum_j \H^{N-1}(\Gamma^+_{E_j}\cup \Gamma^-_{E_j}) + \frac \eps{2^{j+1}} <+\infty\,.
\]
Hence, since $\Chi{E_j}\to \Chi F$ in $L^1$, it is possible to apply Proposition~\ref{genprop} to the sequence $\{E_j\}$, getting that $P(F)=\lim_j P(E_j)$. Thus, (\ref{eq:A.3}) is given by~(\ref{quindici}) and~(\ref{quindici}').
\end{proof}

\section{Proof of Theorem~\mref{approx1} and Theorem~\mref{approx2} }\label{sec:pfThBeC}

This section is devoted to presenting the proofs of Theorems~\mref{approx1} and~\mref{approx2}. Even though Theorem~\mref{approx1} is a particular case of Theorem~\mref{approx2}, we start with the first one, in order to describe the needed arguments in a simpler setting, and then we pass to the general case. As usual, many of the arguments will be exactly the same, up to straightforward modifications in the notation, while others will require some particular care and different proofs.

The plan of this section is the following. First of all, in Section~\ref{subsectthb} we will show how to obtain Theorem~\mref{approx1} using Theorem~\mref{nobound}; then, in Section~\ref{sec42} we will consider the case of densities. Exactly as Theorem~\mref{nobound} was the main brick to get Theorem~\mref{approx1}, the main brick to obtain Theorem~\mref{approx2} will be the generalisation of Theorem~\mref{nobound} to the case of densities, which is Theorem~\ref{genA}.

\subsection{The proof of Theorem~\mref{approx1}\label{subsectthb}}

Having already the result of Theorem~\mref{nobound}, we can reduce to considering sets $E$ such that $\H^{N-1}\big(\partial^* E\cap\partial^* \Omega\big)=0$, and this will make things much simpler. We start with a standard observation (see for instance~\cite{AFP}, or~\cite{M.book}): let $E$ be a set with finite perimeter and for $\delta>0$ let us define $\varphi_\delta = (\Chi{E} \Chi{B_{1/\delta}})\ast \rho_\delta$, where $\rho_\delta$ is a standard mollifier and $B_{1/\delta}$ is the ball with radius $1/\delta$ centered at the origin. In addition, set $G=\big\{x\in\R^N: \varphi_\delta(x) > t\big\}$. Then, for any given $\eps>0$, we can select a sufficiently small $\delta$ and a suitable $t\in (1/10, 9/10)$ so that $G$ is a smooth bounded set such that
\begin{align}\label{standapp}
\big|G \triangle E\big| < \eps\,, && \big|P(G)-P(E)\big| < \eps\,.
\end{align}
Having the smooth approximation, we can clearly pass to a polyhedral one, that is, a bounded polyhedron $G$ such that the above estimates hold. Therefore, what makes the claim of Theorem~\mref{approx1} non-obvious is the presence of the set $\Omega$. In particular, the problem is not to get~(\ref{eq:B3}), but to get the first inequality in~(\ref{eq:B1}). More precisely, we can immediately notice the following standard fact.

\begin{lemma}
Let $G$ be a smooth or polyhedral set in $\R^N$. Then, there is a family $(0,1] \ni \sigma\mapsto G_\sigma$ of smooth or polyhedral sets such that
\begin{align*}
\big| G_\sigma \triangle G\big| \freccia{\sigma\to 0} 0\,, &&
\big|P(G_\sigma)-P(G)\big|\freccia{\sigma\to 0} 0\,, &&
\H^{N-1}\big(\partial^* G_\sigma \cap \partial^* G_{\sigma'}\big)=0 \quad\forall\, \sigma,\,\sigma'>0\,.
\end{align*}
\end{lemma}

An obvious consequence of the above fact is a short proof of Theorem~\mref{approx1} without requiring that $P(F;\Omega)$ is close to $P(E;\Omega)$.

\begin{corollary}\label{almostall}
The claim of Theorem~\mref{approx1} with~(\ref{eq:B1}) replaced by the sole estimate $|P(F)-P(E)|<\eps$ holds.
\end{corollary}
\begin{proof}
Let $E$ be given, and let $G$ be a smooth or polyhedral set such that~(\ref{standapp}) holds. Let then $\sigma\mapsto G_\sigma$ be a family as in the previous lemma. Since $\H^{N-1}\big(\partial^* G_\sigma \cap \partial^* G_{\sigma'}\big)=0$ for every $\sigma,\,\sigma'$ and since $\Omega$ has finite perimeter, we have that
\[
\H^{N-1}\big(\partial^* G_\sigma\cap \partial^* \Omega\big)=0
\]
for all but countably many values of $\sigma$. We can then define $F=G_\sigma$ for a sufficiently small $\sigma$ such that the above equality holds, and we have the validity of~(\ref{eq:B3}), (\ref{eq:B2}), and $|P(F)-P(E)|<\eps$.
\end{proof}

Thus, the above construction already provides an approximation which complies with almost all the requests of Theorem~\mref{approx1}; however, it gives no insight about $P(F;\Omega)$. Nevertheless, we can estimate this quantity as follows. Notice that we are still not assuming that the reduced boundary $\partial^*\Omega$ coincides $\H^{N-1}$-a.e. with the topological boundary $\partial \Omega$, but the latter one is used in the estimate~(\ref{limsupineq}) below.

\begin{lemma}\label{almapp1}
Let $\Omega$ and $E$ be two sets of finite perimeter. Then, there exists a sequence of smooth or polyhedral sets $\{E_j\}$ such that
\begin{align}\label{yeth}
P(E_j)\to P(E)\,, && |E_j\triangle E|\to 0\,, && \H^{N-1}(\partial^* E_j\cap\partial^* \Omega)=0\quad \forall\, j\in\N\,,
\end{align}
and with the property that
\begin{equation}\label{limsupineq}
\limsup_{j\to\infty} \big|P(E_j;\Omega) - P(E;\Omega)\big| \leq  \H^{N-1}\big(\partial^* E \cap \partial \Omega\big)\,.
\end{equation}
\end{lemma}
\begin{proof}
The existence of a sequence $\{E_j\}$ satisfying~(\ref{yeth}) has been already established in Corollary~\ref{almostall}. We show now that any such sequence also satisfies~(\ref{limsupineq}).\par

For any $\eta>0$, let us call $\Gamma_\eta = \big\{x\in\R^N: {\rm dist}(x,\partial\Omega)<\eta\big\}$, which is an open set. Since $\partial\Omega= \cap_{\eta>0} \Gamma_\eta$ because the topological boundary $\partial\Omega$ is closed, we can select a sufficiently small $\eta$ such that
\[
\H^{N-1}\big(\partial^* E \cap \Gamma_{2\eta}\big) < \H^{N-1}\big(\partial^* E \cap \partial \Omega\big)+\eps\,.
\]
We notice now that, writing as usual $\Omega^{(1)}$ to denote the points of density $1$ of $\Omega$, the set $U=\Omega^{(1)}\setminus \overline{\Gamma_\eta}$ is open. Indeed, calling $A$ the interior of $\Omega^{(1)}$, it is enough to observe that $\Omega^{(1)}\setminus \overline{\Gamma_\eta}=A\setminus\overline{\Gamma_\eta}$, since the latter set is surely open. And in turn, if $x \in \Omega^{(1)}\setminus A$, then there exists some $y\in \partial^*\Omega\subseteq\partial\Omega$ such that $|y-x|<\eta$, and then $x\in \Gamma_\eta$ and thus $x\notin\Omega^{(1)}\setminus \overline{\Gamma_\eta}$. Analogously, the set $V = \Omega^{(0)}\setminus\overline{\Gamma_\eta}$ is open. Since $\R^N\setminus (\Omega^{(0)}\cup\Omega^{(1)})\subseteq\partial\Omega$, we deduce that $\R^N$ is the union of the three open sets $U,\, V$ and $\Gamma_{2\eta}$.\par

Since $E_j$ converges in the $L^1$ sense to $E$, the lower semicontinuity of the perimeter together with the fact that $U$ and $V$ are open implies that $P(E;U)\leq \liminf_j P(E_j; U)$ and $P(E;V)\leq \liminf_j P(E_j;V)$, and then there exists a sufficiently large $j$ such that
\begin{align*}
P(E_j; U) \geq P(E; U) -\eps\,, &&
P(E_j; V) \geq P(E; V) -\eps\,.
\end{align*}
By the first convergence in~(\ref{yeth}), up to increasing the value of $j$ we have also
\[
P(E_j)\leq P(E)+\eps\,.
\]
Putting together the above estimates, we have then
\[\begin{split}
P(E_j;\Omega) &\leq P(E_j) - P(E_j; V)
\leq P(E) - P(E; V) + 2\eps
\leq P(E; U ) + P(E; \Gamma_{2\eta}) + 2\eps\\
&\leq P(E; \Omega ) + \H^{N-1}\big(\partial^* E \cap \partial \Omega\big) + 3\eps\,.
\end{split}\]
Similarly,
\[\begin{split}
P(E_j;\Omega) &\geq P(E_j;U)
\geq P(E;U) - \eps
\geq P(E) - P(E;V) - P(E; \Gamma_{2\eta})-\eps\\
&\geq P(E) - P(E;V) - \H^{N-1}\big(\partial^* E \cap \partial \Omega\big)-2\eps
\geq P(E;\Omega) - \H^{N-1}\big(\partial^* E \cap \partial \Omega\big)-2\eps\,.
\end{split}\]
The last two estimates give~(\ref{limsupineq}) and the proof is complete.
\end{proof}

We are ready to show Theorem~\mref{approx1}.
\proofof{Theorem~\mref{approx1}}
Let us take $E$ and $\Omega$ as in the statement, and let us fix $\eps>0$. We apply Theorem~\mref{nobound} to the set $E$, getting a set $G$ which satisfies~(\ref{eq:A.1})--(\ref{eq:A.4}) with $G$ in place of $F$. Then, we apply Lemma~\ref{almapp1} to $G$, getting a sequence $\{G_j\}$ satisfying~(\ref{yeth}) and~(\ref{limsupineq}) with $G$ in place of $E$. We can then define $F=G_j$ for a suitable $j$. The last property of~(\ref{yeth}) ensures that~(\ref{eq:B3}) holds whatever $j$ is, while~(\ref{eq:B2}) and the second property of~(\ref{eq:B1}) follow by~(\ref{eq:A.2}), (\ref{eq:A.3}) and the first two properties of~(\ref{yeth}). The first property of~(\ref{eq:B1}) then follows, since~(\ref{limsupineq}) together with~(\ref{eq:A.4}) ensures that $P(G_j;\Omega)\to P(G;\Omega)$, and in turn $P(G;\Omega)=P(E;\Omega)$ by~(\ref{eq:A.1}).
\end{proof}

\begin{remark}\label{laterinterest}
When we defined $\varphi_\delta= (\Chi{E} \Chi{B_{1/\delta}})\ast \rho_\delta$, we have multiplied by $\Chi{B_{1/\delta}}$ only to have a smooth function $\varphi_\delta$ with compact support; this allows us to obtain a set $G$ satisfying~(\ref{standapp}) which is smooth \emph{and bounded}. If, instead, we define $\varphi_\delta=\Chi{E}\ast \rho_\delta$, then everything works exactly in the same way with the only difference that $G$ is smooth but \emph{not necessarily bounded}.
\end{remark}

\subsection{The proof of Theorem~\mref{approx2}\label{sec42}}

Let us now consider the situation when volume and perimeter are no more the standard ones, but the ones given by~(\ref{genvolper}) for two densities $f$ and $g$. The obvious strategy would be to repeat the whole construction in this more general setting; and in fact, this is more or less what we are going to do, but not all the arguments will work in the same way. Most of the effort of this section will be needed to get the generalisation of Theorem~\mref{nobound}, which is Theorem~\ref{genA} below. The construction will follow exactly the same path as in Section~\ref{sec:ThA}, but the notation will be a bit more complicated and the constants will be often slightly different.

Once Theorem~\ref{genA} is at our disposal, we will work on Theorem~\mref{approx2}; the construction to deduce this result from Theorem~\ref{genA} will be similar to the one we used to get Theorem~\mref{approx1} from Theorem~\mref{nobound}, but some modifications will be required. In particular, it is important to notice one thing. Namely, the claim of Theorem~\ref{genA} is exactly the same as that of Theorem~\mref{nobound}, except for the fact that volumes and perimeters are replaced by $f$-volumes and $g$-perimeters; in particular, observe that in Theorem~\mref{nobound} the assumption that $E$ had finite perimeter automatically implied that $E$ (or $\R^N\setminus E$) had finite volume, while this time the assumption that $E$ has finite $g$-perimeter does not imply anything about the Euclidean volume or the $f$-volume of $E$. Instead, the claim of Theorem~\mref{approx2} is not exactly the same as in Theorem~\mref{approx1}, up to the change of notation; indeed, some extra assumption will be needed in order to generalise some parts of the construction. In the following Section~\ref{sec:counterex} we will show through some counterexamples that our extra assumptions are actually needed.

\begin{theorem}[Removing the common boundary in the density case]\label{genA}
Let $f:\R^N\to(0,+\infty)$ be $L^1_{\rm loc}(\R^N)$ and locally bounded, and $g:\R^N\times\S^{N-1}\to (0,+\infty)$ continuous in the first variable, and convex in the second one in the sense of Definition~\ref{1homcon}. Let $\Omega$ and $E$ be two sets of locally finite $g$-perimeter in $\R^N$. Then, for every $\eps>0$ there exists a set $F\subseteq\R^N$ such that
\begin{gather}\label{neweq:A.1}
F\cap \Omega= E\cap \Omega\,,\qquad\partial^* F \cap \Omega=\partial^* E \cap \Omega\,,
\\\label{neweq:A.2}
 |F\triangle E|_f < \eps\,,\\
\label{neweq:A.3}
|P_g(F) - P_g(E)|< \eps\,,\\
\label{neweq:A.4}
\H^{N-1}(\partial^* F \cap \partial^* \Omega)=0\,.
\end{gather}
\end{theorem}
\begin{proof}
We will proceed exactly as we did in Section~\ref{sec:ThA}, but we need to carefully underline all the differences. Let us start with Lemma~\ref{lemma0}: let again $x$ be a point of $(N-1)$-density $1$ for $\Gamma_E^+$, call $\nu=\nu_E(x)$, and fix a small $\delta$. Since $f$ is locally bounded and $g$ is continuous, we fix a constant $M\geq 1$ such that, for each $y$ in a small neighborhood of $x$ and each $\alpha\in\S^{N-1}$, both $f(y)$ and $g(x,\alpha)$ are smaller than $M$. We want to obtain $\bar r$ such that the following extension of~(\ref{uno})--(\ref{quattro}) holds:
\begin{gather*}
|C^+(x,r)\cap (E\cup \Omega)| \leq \delta^2 \,\frac{\omega_{N-1} g(x,\nu)^2 }{M^2}\, r^N \,, \\
|C^-(x,r)\setminus(E\cap \Omega)| \leq \delta^2\, \frac{\omega_{N-1} g(x,\nu)^2 }{M^2}\, r^N\,,\\
\H_g^{N-1}\big(\Gamma^+_E\cap C_{\delta r/M}(x,r)\big) \geq (1-\delta)g(x,\nu)\omega_{N-1}r^{N-1}\,,\\
\H_g^{N-1}\big(\partial^* E \cap C(x,(1+\delta)r)\big) \leq (1+N\delta)g(x,\nu)\omega_{N-1}r^{N-1}\,,\\
\H_g^{N-1}\big(\partial^*\Omega\cap C(x,(1+\delta)r)\big) \leq (1+N\delta)g(x,\nu) \omega_{N-1}r^{N-1}\,.
\end{gather*}
In the above estimates, by $\H^{N-1}_g$ we mean the measure $\H^{N-1}$ with weight $g$: that is, for any set $G$ of finite perimeter and any Borel set $D$ we write
\[
\H^{N-1}_g(\partial^* G\cap D) = \int_{\partial^* G\cap D} g(z,\nu_G(z))\,d\H^{N-1}(z)\,.
\]
The first two estimates can be obtained exactly as we did for~(\ref{uno}) and~(\ref{due}); indeed, the fact that the constant $\delta^2$ has been replaced by $\delta^2 g(x,\nu)^2/M$ makes no difference---we could put there any constant that we want, our choice is the correct one for later use. Concerning~(\ref{tre}), (\ref{treemezzo}) and~(\ref{quattro}), instead, it is immediate to notice that the above estimates are the ``correct'' generalisations, and the proof is exactly the same. In fact, since $x$ is a density point for $\partial^* E$ (and the same holds for $\partial^*\Omega$), in a small cylinder of radius $r$ around $x$ the $\H^{N-1}$ measure of $\partial^* E$ is very close to $\omega_{N-1} r^{N-1}$, and the vast majority of the points have normal vector almost equal to $\nu$. Summarizing, the above generalisation of Lemma~\ref{lemma0} holds. In particular, since $g$ is continuous in the first variable, up to further decreasing $\bar r$ if necessary, we can assume that for every $y\in C(x,\bar r)$ one has
\begin{equation}\label{smalldev}
(1-\delta) g(x,\nu)\leq g(y,\nu)\leq (1+\delta) g(x,\nu)\,.
\end{equation}

Let us now pass to consider Lemma~\ref{lemma1}. Also in this case, we can prove the analogues of~(\ref{sei})--(\ref{nove+}), substituting everywhere $\H^{N-1}$ with $\H^{N-1}_g$, replacing the volume estimate $| \widetilde E \setminus E| \leq 4 \delta \omega_{N-1} r^N$ in~(\ref{sei}) with
\[
\big| \widetilde E \setminus E\big|_f \leq 4 \delta \omega_{N-1} r^N g(x,\nu)\,,
\]
and replacing all the terms $\omega_{N-1} r^{N-1}$ in~(\ref{sette})--(\ref{nove+}) with $\omega_{N-1} r^{N-1} g(x,\nu)$; in addition, the term $5N\delta$ (resp., $-3\delta$) in~(\ref{otto}) (resp., in~(\ref{nove+})) becomes $(5N+1)\delta$ (resp., $-4\delta$), and this time the height $h$ belongs to $(\delta r g(x,\nu)/M, 2\delta r g(x,\nu)/M)$.\par

The proof of these estimates is almost exactly the same as in Lemma~\ref{lemma1}, up to estimating $f$ and $g$ from above with $M$ when needed, and up to replacing the volume and perimeter with the $f$-volume and the $g$-perimeter. More precisely, the starting estimate~(\ref{cinque}) has to be rewritten this time as
\begin{equation}\label{newcinque}
\H^{N-1} (A_h\cup A_{-h})\leq 3\delta\,\frac{ \omega_{N-1}g(x,\nu)}M r^{N-1}\leq 3\delta \omega_{N-1}g(x,\nu) r^{N-1} \,,
\end{equation}
where the second inequality is obvious since $M\geq 1$. Then, the existence of a suitable height $h\in (\delta r g(x,\nu)/M, 2\delta r g(x,\nu)/M)$ satisfying this estimate is the obvious extension of what we did in Lemma~\ref{lemma1}. The validity of the extension of~(\ref{sei}) is again immediate, keeping in mind that $f\leq M$. The validity of the extensions of~(\ref{sette}), (\ref{ottoemezzo}) and~(\ref{nove}) follows with the very same arguments as in Lemma~\ref{lemma1}, thanks to the form of the extensions of~(\ref{tre}), (\ref{treemezzo}) and~(\ref{quattro}).

The argument to get~(\ref{otto}) needs just a little more care; since this time $h\leq 2\delta r g(x,\nu)/M$, then the estimate $\H^{N-1}(Q_{\rm lat})\leq 4\delta(N-1)\omega_{N-1} r^{N-1}$ reads this time as
\[
\H^{N-1}_g(Q_{\rm lat}) \leq M \H^{N-1}(Q_{\rm lat}) \leq 4\delta(N-1)\omega_{N-1} r^{N-1} g(x,\nu)\,,
\]
and in the same way the estimate for $\H^{N-1}_g (Q_{\rm down})$ works since~(\ref{cinque}) has now been substituted by~(\ref{newcinque}). Concerning the estimate for $Q_{\rm up}$, of course one cannot obtain anything better than $\H^{N-1}(Q_{\rm up})\leq \omega_{N-1} r^{N-1}$, which using the inequality $g\leq M$ would give the estimate $\H^{N-1}_g (Q_{\rm up})\leq \omega_{N-1} r^{N-1} M$, which is too weak. Nevertheless, we can observe that for $\H^{N-1}$-a.e. point of $\partial^* \widetilde E\cap Q_{\rm up}$ the normal vector is exactly $\nu$. As a consequence, by~(\ref{smalldev}) we deduce
\[
\H^{N-1}_g (Q_{\rm up})\leq \omega_{N-1} r^{N-1} g(x,\nu) (1+\delta)\,.
\]
We obtain thus the extension of~(\ref{otto}) by repeating the original argument, and due to the above term $(1+\delta)$ the term $5N\delta$ becomes $(5N+1)\delta$. Finally, the argument to get the new version of~(\ref{nove+}) is again the obvious extension of the original argument, again keeping in mind that the normal vector on $Q_{\rm up}$ coincides with $\nu$ almost everywhere; since we use again~(\ref{newcinque}) in place of~(\ref{cinque}), and since we apply also~(\ref{smalldev}), there is an extra term $-\delta$, and thus the new version of~(\ref{nove+}) has the term $1-4\delta$ in place of $1-3\delta$. The extension of Lemma~\ref{lemma1} is then proved.

The extension of the construction leading to Lemma~\ref{34} is very simple, one only needs to consider the quantities $r_j^{N-1} g(x_j,\nu(x_j))$ in place of $r_j$. More precisely, this time in the definition~(\ref{defr0}) of $r^0(x)$ one must replace the estimates~(\ref{uno})--(\ref{quattro}) with their extensions; then, this time the point $x_1$ is chosen so to almost maximize the quantity $r^0(x)^{N-1} g(x,\nu(x))$; then, in the definition~(\ref{defr1}) of $r^1(x)$ one should again use the extensions of~(\ref{uno})--(\ref{quattro}), and again $x_2$ is chosen almost maximizing $r^1(x)^{N-1} g(x,\nu(x))$, and so on. With the obvious modifications, the proof of Lemma~\ref{lm2} still works, and the new version of~(\ref{lemma2}) is
\[
\sum r_j^{N-1} g(x_j,\nu(x_j)) \leq \frac 2{\omega_{N-1}}\, \H^{N-1}_g(\Gamma^+_E)\,.
\]
Thanks to this estimate, the modification of the proof of Lemma~\ref{34} is obvious, and the claim is basically the same; the only difference is that $\H^{N-1}$ must be replaced by $\H^{N-1}_g$ in (the new versions of) the estimates~(\ref{volinf}), (\ref{popepi}) and~(\ref{qual}), and that the terms $1+5N$ and $N+3$ in~(\ref{popepi}) and~(\ref{qual}) are now replaced by $2+5N$ and $N+4$.

Finally, it is now straightforward how to modify the proofs of Proposition~\ref{finally} and~\ref{finally2} and of Theorem~\mref{nobound}. Since the present theorem is exactly the extension of Theorem~\mref{nobound}, the proof is concluded.
\end{proof}

Having this result at hand, we can now present the proof of our last main result. We will basically follow the lines of the proof of Theorem~\mref{approx1}, using~Theorem~\ref{genA} in place of Theorem~\mref{nobound}.

\proofof{Theorem~\mref{approx2}}
Since the claim of Theorem~\mref{approx2} is similar to that of Theorem~\mref{approx1}, except that volumes and perimeters are now with respect to two densities, we will follow the lines of the proof of Theorem~\mref{approx1}. However, the very first thing that we did in Section~\ref{subsectthb} to obtain that proof does not work. Indeed, we started defining the set $G$ as a suitable superlevel of the function $\varphi_\delta = (\Chi{E} \Chi{B_{1/\delta}})\ast \rho_\delta$, getting $|G \triangle E| < \eps$ and $|P(G)-P(E)| < \eps$. The fact that this works is quite standard in the Euclidean case (that is, with $f\equiv 1$ and $g\equiv 1$), but it is easily seen to be false with general $f$ and $g$. In fact, while in the Euclidean case the information that $P(E)<+\infty$ implies that $E$ (or its complement) has finite volume, in the general case the information that $P_g(E)<+\infty$ does not imply in general that $E$ has finite volume, nor $f$-volume. Nevertheless, as already anticipated in Remark~\ref{laterinterest}, ``cutting the set'' multiplying by $\Chi{B_{1/\delta}}$ is done just to obtain a bounded set, but it is not needed otherwise. Therefore, if we define this time $\varphi_\delta = \Chi{E} \ast \rho_\delta$, then as in the previous section we can define $G$ as a suitable superlevel of $\varphi_\delta$ so that
\begin{align*}
\big|G \triangle E\big|_f < \eps\,, && \big|P_g(G)-P_g(E)\big| < \eps\,.
\end{align*}
The set $G$ is now smooth but not necessarily bounded, and then this time passing to a polyhedral set $G$ is not possible. However, the very same construction as in Section~\ref{subsectthb} still works, and it allows to obtain the first part of the claim of the Theorem, that is, the existence of a smooth, possibly unbounded set $F$ satisfying~(\ref{eq:weighted1}), (\ref{eq:weighted2}) and~(\ref{eq:weighted3}).

To conclude the proof, we have then to obtain the existence of such a set $F$ which is smooth and bounded (and then passing to a polyhedral set is immediate). This is of course possible if $E$ is bounded, because so are $G$ and $F$ by construction. Suppose, instead, that we only know that $E$ has finite $f$-volume and there is some constant $M$ such that the inequality $g(x,\nu)\leq Mf(x)$ is true for every $x\in\R^N$ and $ \nu\in\S^{N-1}$. Since we already know how to handle the case of a bounded set, the thesis follows by the above construction if we find some set $H$ such that
\begin{align}\label{endC}
\big| E \triangle H \big|_f < \eps\,, && \H^{N-1}_g\big(\partial^* E \triangle \partial^* H\big) < \eps\,.
\end{align}
This will be achieved simply setting $H = E \cap B_R$ for a suitable large radius $R$. Indeed, the assumption that $E$ has finite $f$-volume, together with the fact that it has finite perimeter, ensures the existence of a suitable, large $R_0$ such that
\begin{align}\label{almendC}
\big| E \setminus B_{R_0}\big|_f < \eps\,, && \H^{N-1}_g\big( \partial^* E \setminus B_{R_0}\big) < \frac\eps 2\,.
\end{align}
Then, the first estimate in~(\ref{endC}) is true for every $R\geq R_0$; moreover, since
\[
+\infty > | E|_f  = \int_{R=0}^{+\infty} \bigg(\int_{E\cap \partial B_R} f(x)\, d\H^{N-1}(x)\bigg)\, dR\,,
\]
we deduce the existence of some $R\geq R_0$ such that
\[
\int_{E\cap \partial B_R} f(x)\, d\H^{N-1}(x) \leq \frac \eps{2M}\,.
\]
Then, the bound of $g$ in terms of $f$ gives
\[
\int_{E\cap\partial B_R} g\bigg(x, \frac x{|x|}\bigg)\, d\H^{N-1}(x)\leq 
M \int_{E\cap\partial B_R} f(x)\, d\H^{N-1}(x)\leq \frac \eps 2\,.
\]
Adding this estimate with the second one in~(\ref{almendC}) we obtain for this value of $R$ the validity of~(\ref{endC}). The proof is then concluded.
\end{proof}

\section{Sharpness of the assumptions\label{sec:counterex}}

This last section is devoted to discussing the sharpness of the assumptions of our main theorems. Theorem~\mref{nobound} has no assumption at all on $\Omega$ and $E$, so there is nothing to discuss. Concerning Theorem~\mref{approx1}, we only ask $\Omega$ to be a set of finite perimeter such that $\H^{N-1}(\partial\Omega\setminus\partial^* \Omega)=0$. This is clearly an extremely weak assumption, in particular compared with the Lipschitz assumption of the previous results of this kind, so we do not try to sharpen it. Thus, the only interesting claim to consider is that of Theorem~\mref{approx2}.

Let us first consider the assumptions on $g$; the convexity in the second variable is standard, and it is necessary to give a meaning to the perimeter with density. Instead, the continuity in the first variable is not a very weak assumption, and, at first sight, one might think that it is not sharp. Instead, it is simple to realize that this assumption is needed, as the following example shows.

\begin{example}
Let us consider the case when $\Omega$ is the unit ball, $f\equiv 1$, and $g(x,\nu)$ equals $1$ whenever $|x|=1$, and $2$ otherwise. Then, $g$ is not continuous in the first variable, and it is easy to observe that Theorem~\mref{approx2} does not hold. Indeed, if $E$ is a set of finite perimeter with $\H^{N-1}(\partial^* E \cap \partial^* \Omega)>0$, then for every set $F$ the validity of~(\ref{eq:weighted3}) prevents the validity of~(\ref{eq:weighted1}) and~(\ref{eq:weighted2}) if $\eps$ is small enough. Alternatively, we can consider $g(x,\nu)=1$ for $|x|\leq 1$ and $2$ otherwise, so that $g$ is not continuous in the first variable even up to possible transformations in a negligible set. In this case, we get a counterexample to Theorem~\ref{genA}.
\end{example}

This example shows that the assumption of the first part of the claim of Theorem~\mref{approx2} are sharp; thus, let us consider the second part of the claim. One might hope that the boundedness of both the $f$-volume and the $g$-perimeter of $E$ is enough to get an approximating bounded set $F$, without the extra assumption that $g \lesssim f$. Instead, the following example shows that also this last assumption is needed.

\begin{example}\label{<<>>}
Let us consider a $2$-dimensional situation, defining the sets $E$ and $\Omega$ as
\[
E=\Omega=\Big\{(x,y)\in\R^2: x\geq 1,\, |y|\leq x^{-2} \Big\}\,.
\]
The density $f$ for the volume is simply the Euclidean density $f\equiv 1$, so that $|E|_f<+\infty$. The density for the perimeter, instead, will be a smooth density depending only on the point and not on the direction; that is, for every $p\in\R^2$ and every $\nu,\, \nu' \in \S^1$ we have $g(p,\nu)=g(p,\nu')$. For simplicity of notation, we will simply write $g(p)$ to denote this value. The function $g$ will be any positive, smooth function such that for every $x\geq 1$ one has
\begin{align}\label{essex2}
g\big(x,x^{-2}\big) = \frac 1{x^2}\,, && g(x,y) = e^x \quad \hbox{whenever } |y|\leq \frac 1{2x^2}\,.
\end{align}
Notice that the first property implies that $P_g(E)<+\infty$. We are going to show that there is no bounded set $F$ satisfying~(\ref{eq:weighted1}) and~(\ref{eq:weighted2}) if $\eps$ is small enough, so in particular the bounded smooth or polyhedral approximation is false. This shows the importance of the assumption $g \lesssim f$ in Theorem~\mref{approx2}.
\begin{figure}[htbp]
\begin{tikzpicture}[>=>>>]
\draw (3,1.8) node[anchor=north east] {$E$};
\filldraw[green!40, draw=black,line width=.8pt]  (7,-0.28) .. controls (6,-.32) and (5,-.4) .. (4,-0.5) .. controls (3,-0.66) and (2,-1) .. (1,-2) .. controls (1,0) and (1,0) .. (1,2) .. controls (2,1) and (3,0.66) .. (4,0.5) .. controls (5,0.4) and (6,0.32) .. (7,0.28);
\filldraw[orange!80, draw=black, line width=.8pt]  (7,-0.14) .. controls (6,-.16) and (5,-.2) .. (4,-0.25) .. controls (3,-0.33) and (2,-0.5) .. (1,-1) .. controls (1,0) and (1,0) .. (1,1) .. controls (2,0.5) and (3,0.33) .. (4,0.25) .. controls (5,0.2) and (6,0.16) .. (7,0.14);
\draw (-1,-0.5) node[anchor=east] {$g\gg1$};
\draw[->] (-1,-.5) -- (1.5,0);
\draw (-1,1.5) node[anchor=east] {$g\ll1$};
\draw[->] (-1,1.5) -- (1.5,1.2);
\label{Figure2}
\end{tikzpicture}
\caption{The situation in Example~\ref{<<>>}; the density $g$ is very low near the boundary of $E$, but very large in the orange region.}
\end{figure}
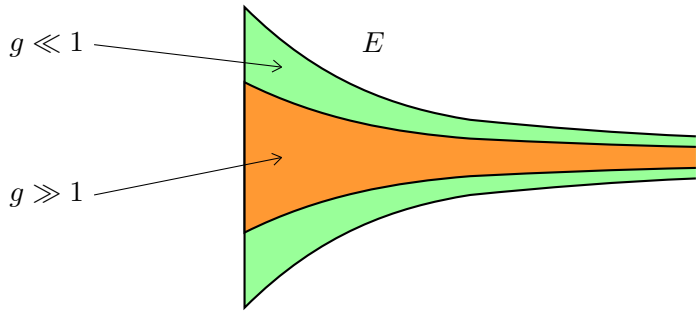
To obtain the thesis, it is enough to notice that, if $F$ is a bounded set satisfying~(\ref{eq:weighted2}) with $\eps$ small enough, then there must be a point in $\partial F$ having coordinates $(t,0)$ for some large $t$, and $t$ is arbitrarily large if $\eps\ll 1$. We claim that any curve joining $(t,0)$ to the boundary of $\big\{ (x,y)\in\R^2: x\geq 1,\, |y|\leq 1/2x^2\big\}$, which is the orange region in Figure~\ref{Figure2}, has very large $g$-length. This immediately implies that $F$ cannot satisfy~(\ref{eq:weighted1}) as well.\par

To show our claim, let $\gamma$ be any curve joining the point $(t,0)$ with some point $(a,b)$ in the boundary of the orange region. Since the (Euclidean) distance of $(t,0)$ from this boundary is larger than $1/2t^2$, then we can call $\gamma^-$ the first part of $\gamma$, up to a (Euclidean) length of $1/2t^2$. Thus, the first coordinate of any point of $\gamma^-$ is larger than $t-1$, thus by the second property in~(\ref{essex2}) we immediately get
\[
\H^1_g(\gamma)\geq \H^1_g(\gamma^-) \geq \frac{e^{t-1}}{2t^2}\,.
\]
Since the last quantity is arbitrarily large as soon as $t$ is large enough, hence as soon as $\eps$ is small enough, the proof is concluded.
\end{example}

\begin{remark}
In the above example, the function $g$ is unbounded both from above and from below; however, what we were looking for was a counterexample with $g$ very large somewhere. As a consequence, one might wish to find a counterexample where $g$ remains bounded from below (away from $0$). It is simple to observe that this is not possible in dimension $2$, because the Euclidean length of any curve is larger than its diameter, and then any connected and unbounded set $E$ has infinite perimeter if $g$ is bounded from below. Instead, in dimension $N\geq 3$ this is easily possible. Actually, an example can be obtained with a very simple modification of Example~\ref{<<>>}: it is enough to consider $\Omega=E$ as the set of points $(x,w)\in \R\times\R^2$ such that $x\geq 1$ and $|w|\leq x^{-1}$, and to modify~(\ref{essex2}) asking $g (x,w)= 1$ whenever $|w|=1/x$, and $g(x,w)=e^x$ whenever $|w|\leq 1/2x$, as well as $g\geq 1$ everywhere. This time $g$ is bounded away from $0$ and the obvious modification of the argument of the above example again shows that $E$ cannot be approximated by bounded sets.
\end{remark}

\end{document}